\documentclass{amsart}

\usepackage{amsmath}
\usepackage{amssymb}
\usepackage{amsthm}
\usepackage{mathscinet}
\usepackage{graphicx}
\usepackage{enumitem}

\theoremstyle{plain}
\newtheorem*{thm*}{Theorem}
\newtheorem{thm}{Theorem}
\newtheorem{prop}[thm]{Proposition}
\newtheorem{lem}[thm]{Lemma}
\newtheorem{cor}[thm]{Corollary}

\theoremstyle{definition}
\newtheorem{defn}[thm]{Definition}
\newtheorem{claim}{Claim}[thm]

\newtheorem{question}{Question}
\newtheorem*{question*}{Question}
\newtheorem{example}{Example}

\newcommand{\intr}{\mathrm{Int}}

\begin{document}

\title[Accessibility of countable sets]{Accessibility of countable sets\\in plane embeddings of arc-like continua}
\author{Ana Anu\v{s}i\'{c} \and Logan C. Hoehn}
\date{\today}

\address{Nipissing University, Department of Computer Science \& Mathematics, 100 College Drive, Box 5002, North Bay, Ontario, Canada, P1B 8L7}
\email{anaa@nipissingu.ca}
\email{loganh@nipissingu.ca}

\thanks{This work was supported by NSERC grant RGPIN-2019-05998}

\subjclass[2020]{Primary 54F15, 54C25; Secondary 54F50, 37B45, 37E30}
\keywords{Plane embeddings; accessible point; arc-like continuum; Knaster continuum; attractor}

\begin{abstract}
We consider the problem of finding embeddings of arc-like continua in the plane for which each point in a given subset is accessible.  We establish that, under certain conditions on an inverse system of arcs, there exists a plane embedding of the inverse limit for which each point of a given countable set is accessible.  As an application, we show that for any Knaster continuum $K$, and any countable collection $\mathcal{C}$ of composants of $K$, there exists a plane embedding of $K$ in which every point in the union of the composants in $\mathcal{C}$ is accessible.  We also exhibit new embeddings of the Knaster buckethandle continuum $K$ in the plane which are attractors of plane homeomorphisms, and for which the restriction of the plane homeomorphism to the attractor is conjugate to a power of the standard shift map on $K$.
\end{abstract}

\maketitle

\section{Introduction}
\label{sec:intro}

In this paper we study a generalized version of the Nadler-Quinn problem, which was recently solved in the affirmative by the authors and A.\ Ammerlaan in \cite{ammerlaan-anusic-hoehn2024-2}.  The Nadler-Quinn problem \cite{nadler1972} asked whether for every arc-like continuum $X$, and every $x \in X$, there exists a plane embedding $\Omega \colon X \to \mathbb{R}^2$ of $X$ such that $\Omega(x)$ is an accessible point of $\Omega(X)$.  Recall that a continuum is called \emph{arc-like} if it is an inverse limit on arcs, and a point $x$ of a plane continuum $X$ is called \emph{accessible} if there is an arc $A$ in the plane such that $A \cap X = \{x\}$.

In \cite{ammerlaan-anusic-hoehn2023}, we define a ``radial departure'' of an interval map with respect to a point, and then show that if no bonding map has a negative radial departure with respect to $x$, then there is a plane embedding $\Omega \colon X \to \mathbb{R}^2$ of $X$ such that $\Omega(x)$ is an accessible point of $\Omega(X)$ (a similar condition was given in \cite{anusic-bruin-cinc2020}, using the term ``zigzags'').  Then, in \cite{ammerlaan-anusic-hoehn2024,ammerlaan-anusic-hoehn2024-2} we show that, given an arc-like continuum $X$ and $x \in X$, it is possible to find an appropriate inverse limit representation of $X$ such that the bonding maps do not have negative radial departures with respect to $x$, thus solving the Nadler-Quinn problem.

We can naturally generalize the problem as follows.  Given an arc-like continuum $X$ and a finite, or countably infinite, set $Z \subset X$, is there a plane embedding $\Omega \colon X \to \mathbb{R}^2$ such that each point of $\Omega(Z)$ is an accessible point of $\Omega(X)$?  In this paper, we give conditions (see Proposition~\ref{prop:embed ctble set accessible}) on the bonding maps of an inverse system on arcs which allow an existence of such an embedding of the inverse limit, phrased in terms of ``removable visors'', which generalize the ``no negative radial departures'' condition from \cite{ammerlaan-anusic-hoehn2023}.

We apply the results to the particular case where $X$ is a Knaster continuum, i.e.\ an inverse limit on intervals in which the bonding maps are open (and non-monotone).  It is known that there are uncountably many non-equivalent embeddings of each Knaster continuum in the plane \cite{mayer1983}, where two plane embeddings $\Omega_1,\Omega_2$ are \emph{equivalent} if there exists a plane homeomorphism $H \colon \mathbb{R}^2 \to \mathbb{R}^2$ such that $\Omega_2 = H \circ \Omega_1$.  Also, for any composant $\kappa$ of a Knaster continuum $K$, there exists a plane embedding $\Omega \colon K \to \mathbb{R}^2$ of $K$ for which $\Omega(\kappa)$ is \emph{fully accessible} \cite{mahavier1989,  anusic-bruin-cinc2017}, meaning that each point in $\Omega(\kappa)$ is an accessible point of $\Omega(K)$.  In \cite{debski-tymchatyn1993}, a complete classification of accessible sets and prime ends of certain, so-called ``regular'', embeddings of Knaster continua is given.  In such plane embeddings, there are at most two composants which are fully accessible.  Consequently, the following question was asked:

\begin{question*}[Question E. in \cite{debski-tymchatyn1993}]
Does there exist an embedding of a Knaster continuum in the plane for which the prime end structure is different than one of the possible structures in regular embeddings?
\end{question*}

We answer this question in the affirmative in Section~\ref{sec:knaster embeddings}, by giving plane embeddings of Knaster continua for which the set of accessible points is larger than those for the regular embeddings of \cite{debski-tymchatyn1993}.  In fact, we produce embeddings of Knaster continua in the plane for which the set of accessible points is ``as large as possible'', in the following sense.

In \cite[Lemma A]{debski-tymchatyn1993}, the following result is proved: Let $K \subset \mathbb{R}^2$ be a Knaster continuum embedded in the plane, and let $x \in K$ be a point which is not an endpoint of $K$.  If $x$ is an accessible point of $K$, then the composant of $x$ is fully accessible.  On the other hand, Mazurkiewicz \cite{mazurkiewicz1929} proved that if $X \subset \mathbb{R}^2$ is any indecomposable continuum embedded in the plane, then there are at most countably many composants of $X$ which contain more than one accessible point.  Therefore, since any Knaster continuum has at most two endpoints, it follows that the ``largest'' possible set of accessible points for a Knaster continuum in the plane is the union of countably infinitely many composants.  In Theorem~\ref{thm:knaster ctble composants}, we prove that there do in fact exist embeddings for which the set of accessible points has this form; in fact, we show that for any countable collection of composants of any Knaster continuum $K$, there exists a plane embedding of $K$ in which each composant in the collection is fully accessible.

We further focus on the special case of the \emph{bucket-handle} continuum, which is the inverse limit on the interval with a single bonding map which is the tent map $T_2$, as in Definition~\ref{def:Tm} below.  It is well-known (the original construction is due to Barge and Martin \cite{barge-martin1990}, using Brown's theorem \cite{brown1960}, see also \cite{boyland-carvalho-hall2021}) that, when an arc-like continuum $X$ is represented as an inverse limit with a single bonding map $f \colon I \to I$, there is a plane embedding $\Omega \colon X \to \mathbb{R}^2$ for which $\Omega(X)$ is the global attractor of a plane homeomorphism $H \colon \mathbb{R}^2 \to \mathbb{R}^2$, such that $H {\restriction}_{\Omega(X)}$ is conjugate to the shift homeomorphism  $\widehat{f} \colon X \to X$,
\[ \widehat{f}(\langle x_1,x_2,\ldots \rangle) = \langle f(x_1),x_1,x_2,\ldots \rangle .\]

In Example~\ref{ex:knaster attractor}, we show that for every $n \in \mathbb{N}$, there is a plane embedding $\Omega \colon K \to \mathbb{R}^2$ of the bucket-handle continuum $K$ such that $\Omega(K)$ has at least $n$ fully-accessible composants.  The embeddings are constructed using a slight variation of the Barge-Martin construction \cite{barge-martin1990}, so that there is a plane homeomorphism $H \colon \mathbb{R}^2 \to \mathbb{R}^2$ such that $\Omega(K)$ is the global attractor of $H$, and $H {\restriction}_{\Omega(K)}$ is conjugate to the $k$-th iterate of the standard shift-homeomorphism on $K$, for some $k > 1$.  It is important to note that $k > 1$ here.  There are, up to equivalence, only two known plane embeddings $\Omega \colon K \to \mathbb{R}^2$ for which $\Omega(K)$ is the global attractor of a plane homeomorphism $H$ such that $H {\restriction}_{\Omega(K)}$ is conjugate to the standard shift-homeomorphism on $K$; one is orientation-preserving (studied in detail in \cite{brucks-diamond1995}), and one is orientation-reversing (see \cite{bruin1999}).  Phil Boyland has asked\footnote{Personal communication.} whether those are the only two such embeddings, up to equivalence.

The paper is organized as follows.  After giving preliminary definitions in Section~\ref{sec:prelim}, we introduce the notion of (removable) visors in Section~\ref{sec:visors}, and prove the general embedding result for arc-like continua, Proposition~\ref{prop:embed ctble set accessible}.  After recalling some basic properties on Knaster continua in Section~\ref{sec:knaster}, we find alternative inverse limit representations of Knaster continua in Section~\ref{sec:tent factorization}, which are used in Section~\ref{sec:knaster embeddings} to prove Theorem~\ref{thm:knaster ctble composants}.  Finally, we present our new dynamical embeddings of the bucket-handle continuum in Example~\ref{ex:knaster attractor}.  The last section of the paper includes discussion and open questions.

The authors would like to thank Wayne Lewis for helpful discussions on topics related to the subject of this paper.

\section{Preliminaries}
\label{sec:prelim}

By a \emph{map} we mean a continuous function.  We denote by $I$ the unit interval $[0,1]$.  An \emph{arc} is a space homeomorphic to $I$.  If $f \colon I \to I$ is a map, we call it \emph{piecewise-linear} if there exists $n \in \mathbb{N}$ and $0 = c_0 < c_1 < \cdots < c_n = 1$ such that $f$ is linear on $[c_{i-1},c_{i}]$ for each $i \in \{1,\ldots,n\}$.

A \emph{continuum} is a compact, connected, metric space.  A continuum is called \emph{non-degenerate} if it contains more than one point.  A \emph{subcontinuum} of a continuum $X$ is a subspace of $X$ which is a continuum.  A subcontinuum of $X$ is \emph{proper} if it not equal to $X$.  A continuum is \emph{indecomposable} if it is not a union of two proper, non-degenerate subcontinua.  Let $X$ be a continuum and $x \in X$.  The \emph{composant} of $x$ in $X$ is the union of all proper subcontinua of $X$ which contain $x$.  If $X$ is indecomposable, then it has uncountably many composants, which are mutually disjoint, and each of them is dense in $X$.

Let $X \subset \mathbb{R}^2$ be a plane continuum and $x \in X$.  We say that $x$ is \emph{accessible} if there is an arc $A \subset \mathbb{R}^2$ such that $A \cap X = \{x\}$.

For a sequence $X_n$, $n \geq 1$, of continua, and a sequence $f_n \colon X_{n+1} \to X_{n}$, $n \geq 1$, of maps, the sequence
\[ \left \langle X_{n},f_n \right \rangle_{n \geq 1} = \left \langle X_1,f_1,X_2,f_2,X_3,\ldots \right \rangle \]
is called an \emph{inverse system}.  The spaces $X_n$ are called \emph{factor spaces}, and the maps $f_n$ are called \emph{bonding maps}.  The \emph{inverse limit} of an inverse system $\left \langle X_n,f_n \right \rangle_{n \geq 1}$ is the space
\[ \varprojlim \left\langle X_n,f_n \right\rangle_{n \geq 1} = \{(x_1,x_2,\ldots): f_n(x_{n+1}) = x_n \textrm{ for each } n \geq 1\} \subseteq \prod_{n \geq 1} X_n ,\]
equipped with the subspace topology inherited from the product topology on $\prod_{n \geq 1} X_n$.  It is also a continuum.  When the index set is clear, we will use a shortened notation $\varprojlim \left\langle X_n,f_n \right\rangle$.  If all factors spaces are equal to a single continuum $X$, and all bonding maps are equal to a single map $f \colon X \to X$, the inverse limit will be denoted by $\varprojlim \left\langle X,f \right\rangle$.  Note that the inverse limit representation of a continuum is not unique, for example:
\begin{itemize}
\item (Dropping finitely many coordinates) If $n_0 \geq 1$, then $\varprojlim \left\langle X_n,f_n \right\rangle_{n \geq 1} \approx \varprojlim \left\langle X_n,f_n \right\rangle_{n \geq n_0}$.
\item (Composing bonding maps) If $1 = n_1 < n_2 < n_3 < \ldots$ is an increasing sequence of integers, then $\varprojlim \left\langle X_n,f_n \right\rangle_{n \geq 1} \approx \varprojlim \left\langle X_{n_i},f_{n_i} \circ \cdots \circ f_{n_{i+1}-1} \right\rangle_{i \geq 1}$.
\end{itemize}

A map $f \colon X \to X$ is a \emph{near-homeomorphism} if it is a uniform limit of a sequence of homeomorphisms of $X$ to itself.  According to a result of Brown \cite{brown1960}, for any sequence $f_n \colon X \to X$, $n \geq 1$, of near-homeomorphisms, the inverse limit $\varprojlim \left\langle X,f_n \right\rangle$ is homeomorphic to $X$.

A continuum $X$ is called \emph{arc-like} if $X \approx \varprojlim \left\langle I,f_n \right\rangle$ for some sequence of maps $f_n \colon I \to I$, $n \geq 1$.  By \cite{brown1960}, we can without loss of generality assume that these maps $f_n$ are all piecewise-linear.  Every arc-like continuum can be embedded in the plane \cite{bing1951}.

\section{Visors and plane embeddings}
\label{sec:visors}

Throughout this section, let $f \colon I \to I$ be a piecewise-linear map, and let $Z = \{z_1,\ldots,z_n\} \subset I$ be a set of $n$ points such that
\[ z_1 < z_2 < \cdots < z_n  \quad\textrm{and}\quad  f(z_1) < f(z_2) < \cdots < f(z_n) .\]

\begin{defn}
\label{def:visor}
Given $j \in \{1,\ldots,n\}$, a \emph{$Z$-visor} for $z_j$ under $f$ is a point $v \in I$ such that $v < z_j$, $[v,z_j) \cap Z = \emptyset$, and $f(v) > f(z_j)$.  Briefly, we call a point $v \in I$ a $Z$-visor (or just a visor) under $f$ if there exists $j \in \{1,\ldots,n\}$ such that $v$ is a $Z$-visor for $z_j$ under $f$.

Let $a,b,c \in I$, and let $v$ be a visor for $z_j$ under $f$.  We say the triple $\langle a,b,c \rangle$ \emph{removes} the visor $v$ if:
\begin{enumerate}
\item $a < v < b$ and $z_j < c$;
\item $(a,b) \cap Z = \emptyset$;
\item Either $a = 0$ and $0 \notin Z$, or $f(a) \leq f(x)$ for all $x \in [a,b]$;
\item Either $c = 1$, or $f(c) \geq f(x)$ for all $x \in [a,b]$; and
\item $f(b) \leq f(x)$ for all $x \in [b,c]$.
\end{enumerate}
The visor $v$ is called \emph{removable} if there exist $a,b,c \in I$ such that $\langle a,b,c \rangle$ removes $v$.
\end{defn}

Figure~\ref{fig:visors} depicts some examples of removable and non-removable visors.

\begin{figure}
\begin{center}
\includegraphics{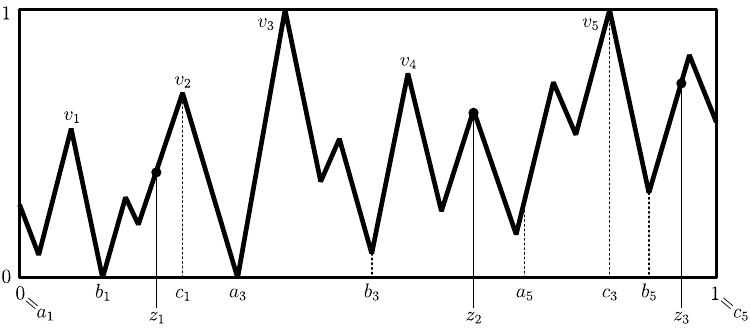}
\end{center}

\caption{An example of a map $f \colon I \to I$ and a set $Z = \{z_1,z_2,z_3\}$, with five $Z$-visors $v_1,\ldots,v_5$ displayed.  For each $i \in \{1,3,5\}$, $v_i$ is removable, and a triple $\langle a_i,b_i,c_i \rangle$ which removes $v_i$ is displayed.  The visors $v_2$ and $v_4$ are not removable.}
\label{fig:visors}
\end{figure}

One sufficient condition for a $Z$-visor $v$ under $f$ to be removable is if there exists $a,b \in I$ such that $a < v < b$, $(a,b) \cap Z = \emptyset$, and $f(a) = f(b) = 0$.  Indeed, in this case, the triple $\langle a,b,1 \rangle$ removes $v$.  This sufficient condition will be all that is needed for the examples we present in Section~\ref{sec:knaster embeddings} below.

We pause briefly to compare the concept of visors to that of negative radial departures, defined in \cite{ammerlaan-anusic-hoehn2023}.  In that paper, the focus was on maps $f \colon [-1,1] \to [-1,1]$ with $f(0) = 0$.  A \emph{negative radial departure} for such a map $f$ is a pair $\langle x_1,x_2 \rangle$ with $x_1 < 0 < x_2$ and $f((x_1,x_2)) = (f(x_2),f(x_1))$.  If we momentarily adapt, in a natural way, the definition of visors above to such a map $f \colon [-1,1] \to [-1,1]$, then it can easily be seen that there exists a non-removable $\{0\}$-visor under $f$ if and only if there exists a negative radial departure of $f$.

\begin{lem}
\label{lem:av bv}
Let $f$ and $Z$ be as above.  Let $v$ be a visor under $f$ which is removable.  There is a unique minimal interval $[a_v,b_v]$ for which there exists $c$ such that $\langle a_v,b_v,c \rangle$ removes $v$.
\end{lem}

\begin{proof}
Let $j \in \{1,\ldots,n\}$ be such that $v$ is a visor for $z_j$ under $f$.  Let $r = \max \{x < v: f(x) = f(z_j)\}$ or $r = 0$ if no such $x < v$ exists, let $s = \min \{x > v: f(x) = f(z_j)\}$, and let $t = \min \{x > z_j: f(x) = f(v)\}$ or $t = 1$ if no such $x > z_j$ exists.  Then it is easy to see that a triple $\langle a,b,c \rangle$ removes $v$ if and only if
\begin{enumerate}[label=(\arabic{*}$'$)]
\item $a \leq r$, $s \leq b$, and $t \leq c$,
\end{enumerate}
and also (2), (3), (4), and (5) from Definition~\ref{def:visor} hold.  In this form, it is evident that the set $\{\langle a,b,c \rangle \in I^3: \langle a,b,c \rangle$ removes $v\}$ is a closed subset of $I^3$.  Hence, its projection
\[ S = \{\langle a,b \rangle \in I^2: \textrm{ there exists $c \in I$ such that $\langle a,b,c \rangle$ removes $v$} \} \]
is closed in $I^2$.  To prove the Lemma, it suffices to prove that this set $S$ is ``directed'', in the sense that for any $\langle a_1,b_1 \rangle, \langle a_2,b_2 \rangle \in S$, there exists $\langle a_3,b_3 \rangle \in S$ such that $[a_3,b_3] \subseteq [a_1,b_1] \cap [a_2,b_2]$.

Let $\langle a_1,b_1 \rangle, \langle a_2,b_2 \rangle \in S$, and suppose that $a_1 < a_2$ and $b_1 < b_2$.  Let $c_1,c_2 \in I$ be such that $\langle a_1,b_1,c_1 \rangle$ removes $v$ and $\langle a_2,b_2,c_2 \rangle$ removes $v$.  We claim that both $\langle a_2,b_1,c_1 \rangle$ and $\langle a_2,b_1,c_2 \rangle$ remove $v$.  Indeed, let $c \in \{c_1,c_2\}$.  We have $a_2 < v < b_1$ and $z_j < c$, and $(a_2,b_1) \cap Z = \emptyset$.  Also, $f(a_2) \leq f(x)$ for all $x \in [a_2,b_1] \subseteq [a_2,b_2]$; either $c = 1$ or $f(c) \geq f(x)$ for all $x \in [a_2,b_1] = [a_1,b_1] \cap [a_2,b_2]$; and $f(b_1) \leq f(x)$ for all $x \in [b_1,c]$, since $f(b_1) \leq f(x)$ for all $x \in [b_1,c_1]$ and $f(b_1) \leq f(b_2) \leq f(x)$ for all $x$ in $[b_2,c_2]$.  Thus $\langle a_2,b_1 \rangle \in S$, as desired.
\end{proof}

\begin{lem}
\label{lem:av bv clear}
Let $f$ and $Z$ be as above.  Let $v$ be a visor under $f$ which is removable.  Then (1) either $a_v = 0$ or $f(a_v) = f(b_v)$, and (2) $f(x) > f(b_v)$ for each $a_v < x < b_v$.
\end{lem}

\begin{proof}
Let $c \in I$ be such that $\langle a_v,b_v,c \rangle$ removes $v$.  We claim that $f(b_v) \leq f(x)$ for all $x \in [v,b_v]$.  Indeed, suppose for a contradiction that there exists $x \in [v,b_v)$ such that $f(x) < f(b_v)$.  It is straightforward to see that $\langle a_v,x,c \rangle$ removes $v$.  But this contradicts the minimality of $[a_v,b_v]$.

Now, suppose $a_v \neq 0$.  Since $f(a_v) \leq f(b_v)$ and $f(v) > f(b_v)$, there exists some $a' \in [a_v,v)$ such that $f(a') = f(b_v)$.  If we take $a'$ to be the maximal such value, then $f(a') \leq f(x)$ for all $x \in [a',v]$.  Also, $f(a') = f(b_v) \leq f(x)$ for all $x \in [v,b_v]$ by the above paragraph.  It follows that $\langle a',b_v,c \rangle$ removes $v$.  Therefore, $a' = a_v$ by minimality of $[a_v,b_v]$, and thus $f(a_v) = f(b_v)$.

For the second claim, suppose for a contradiction that $f(x) = f(b_v)$ for some $x \in (a_v,b_v)$.  It is straightforward to see that if $x < v$, then $\langle x,b_v,c \rangle$ removes $v$, and if $x > v$, then $\langle a_v,x,c \rangle$ removes $v$.  In either case, we have a contradiction with the minimality of $[a_v,b_v]$.
\end{proof}

\begin{lem}
\label{lem:av bv no overlap}
Let $f$ and $Z$ be as above.  For any visors $v_1,v_2$ under $f$ which are removable, if $(a_{v_1},b_{v_1}) \cap (a_{v_2},b_{v_2}) \neq \emptyset$ then $[a_{v_1},b_{v_1}] = [a_{v_2},b_{v_2}]$.
\end{lem}

\begin{proof}
Let $v_1,v_2$ be visors under $f$ which are removable, and let $c_1,c_2 \in I$ be such that $\langle a_{v_1},b_{v_1},c_1 \rangle$ removes $v_1$ and $\langle a_{v_2},b_{v_2},c_2 \rangle$ removes $v_2$.  Suppose that $(a_{v_1},b_{v_1}) \cap (a_{v_2},b_{v_2}) \neq \emptyset$.  Without loss of generality, we may assume that $b_{v_1} \leq b_{v_2}$, so that $b_{v_1} \in (a_{v_2},b_{v_2}]$.  If $b_{v_1} < b_{v_2}$, then $f(b_{v_1}) > f(b_{v_2})$ by Lemma~\ref{lem:av bv clear}(2), but this contradicts the fact that $\langle a_{v_1},b_{v_1},c_1 \rangle$ removes $v_1$.  Therefore, $b_{v_1} = b_{v_2}$.

Again without loss of generality, we may assume that $a_{v_1} \leq a_{v_2}$, so that $a_{v_2} \in [a_{v_1},b_{v_1})$.  If $a_{v_1} < a_{v_2}$, then $f(a_{v_2}) > f(b_{v_1})$ by Lemma~\ref{lem:av bv clear}(2).  At the same time, by Lemma~\ref{lem:av bv clear}(1), we have $f(a_{v_2}) = f(b_{v_2}) = f(b_{v_1})$, a contradiction.  Therefore, $a_{v_1} = a_{v_2}$.  Thus, $[a_{v_1},b_{v_1}] = [a_{v_2},b_{v_2}]$.
\end{proof}

The next Lemma is the main technical result of this section, and is the key ingredient in the proofs of Propositions~\ref{prop:embed n pts accessible} and \ref{prop:embed ctble set accessible} below.  In light of the comments above comparing visors and negative radial departures, it may be seen as a strengthening of Lemma~5.1 of \cite{ammerlaan-anusic-hoehn2023}.

We denote by $\mathbb{H}$ the closed right half-plane,
\[ \mathbb{H} = \{\langle x,y \rangle: x \geq 0\} .\]

\begin{lem}
\label{lem:tuck}
Let $f$ and $Z$ be as above.  Suppose that every $Z$-visor under $f$ is removable.  Then for any $\varepsilon > 0$ there exists an embedding $\Phi \colon I \to \mathbb{H}$ such that
\begin{enumerate}
\item $\|\Phi(x) - \langle 0,f(x) \rangle\| < \varepsilon$ for all $x \in I$;
\item $\Phi(z_j) = \langle 0,f(z_j) \rangle$ for each $j \in \{1,\ldots,n\}$, and $\Phi(x) \notin \partial \mathbb{H}$ for all $x \in I \smallsetminus Z$; and
\item $\Phi \left( I \smallsetminus [z_j,z_{j+1}] \right)$ is disjoint from the bounded complementary component of the circle $\Phi \left( [z_j,z_{j+1}] \right) \cup \left( \{0\} \times [f(z_j),f(z_{j+1})] \right)$, for each $j \in \{1,\ldots,n-1\}$.
\end{enumerate}
\end{lem}

\begin{proof}
The proof will proceed in two steps.  In the first part, we work with properties of visors, towards showing (Claim~\ref{claim:choose targets} below) that we can make a coherent choice of ``targets'' for where to place parts of the graph of $f$ which contain visors, to get them out of the way of the points on the graph corresponding to the points $\{z_1,\ldots,z_n\}$.  In the second part, we carry out the actual embedding in the half-plane.  The reader may find it helpful to refer to Figure~\ref{fig:visor tuck} to visualize the key ideas.

\medskip
\noindent \textbf{Step 1.}
Here we prove some properties of visors under $f$, under the assumption that all visors are removable.

\begin{claim}
\label{claim:target}
Let $v$ be a visor under $f$.  There exists a value $c$ such that $\langle a_v,b_v,c \rangle$ removes $v$, and $c > z_{j'}$ for all $j' \in \{1,\ldots,n\}$ such that $f(v) > f(z_{j'})$.
\end{claim}

Call such a value $c$ a \emph{target} for $v$.

\begin{proof}[Proof of Claim~\ref{claim:target}]
\renewcommand{\qedsymbol}{\textsquare (Claim~\ref{claim:target})}
Let $j \in \{1,\ldots,n\}$ be such that $v$ is a visor for $z_j$ under $f$.  Let $c \in I$ be maximal such that $\langle a_v,b_v,c \rangle$ removes $v$.  Suppose for a contradiction that $c \leq z_{j'}$ for some $j' \in \{1,\ldots,n\}$ such that $f(v) > f(z_{j'})$.  We may assume that $z_{j'-1} < c \leq z_{j'}$.  Clearly $c \neq z_{j'}$ since $f(c) \geq f(v) > f(z_{j'})$.  Now $c$ is a visor for $z_{j'}$ under $f$.  Since, by assumption, all visors are removable, there exists a triple $\langle a',b',c' \rangle$ which removes $c$.  We claim that $\langle a_v,b_v,c' \rangle$ removes $v$.  Indeed, $a_v < v < b_v$ and $z_j < z_{j'} < c'$, and also $(a_v,b_v) \cap Z = \emptyset$.  Also, either $a_v = 0$ and $0 \notin Z$, or $f(a_v) \leq f(x)$ for all $x \in [a_v,b_v]$; either $c' = 1$, or $f(c') \geq f(c) \geq f(x)$ for all $x \in [a_v,b_v]$; and $f(b_v) \leq f(x)$ for all $x \in [b_v,c']$, since $f(b_v) \leq f(x)$ for all $x \in [b_v,c]$, $f(b_v) \leq f(a') \leq f(x)$ for all $x \in [a',b']$, and $f(b_v) \leq f(b') \leq f(x)$ for all $x \in [b',c']$.  But this contradicts the maximality of $c$.
\end{proof}

Choose a finite collection $\mathcal{V}$ of visors under $f$ such that (1) every visor under $f$ is contained in $\bigcup_{v \in \mathcal{V}} [a_v,b_v]$, (2) the intervals $[a_v,b_v]$, $v \in V$, are all distinct, and (3) for each $v \in \mathcal{V}$, $f(v) = \max \{f(x)$: $x \in [a_v,b_v]\}$.  Note that by Lemma~\ref{lem:av bv no overlap}, if $u,v \in \mathcal{V}$, then $u < v$ if and only if $b_u \leq a_v$.

\begin{claim}
\label{claim:target order}
Let $v_1,v_2 \in \mathcal{V}$ such that $f(v_1) \geq f(v_2)$, and let $c_1$ be a target for $v_1$.
\begin{enumerate}
\item If $v_1 < v_2$, then there exists a target $c_2$ for $v_2$ such that $c_2 \leq c_1$.
\item If $v_2 < v_1$, then either there exists a target $c_2$ for $v_2$ such that $c_2 < a_{v_1}$, or $c_1$ is a target for $v_2$.
\end{enumerate}
\end{claim}

\begin{proof}[Proof of Claim~\ref{claim:target order}]
\renewcommand{\qedsymbol}{\textsquare (Claim~\ref{claim:target order})}
Let $j_2 \in \{1,\ldots,n\}$ be such that $v_2$ is a visor for $z_{j_2}$ under $f$.  Since $f(v_1) \geq f(v_2) > f(z_{j_2})$, by definition of a target, we must have $z_{j_2} < c_1$.  Also, $f(c_1) \geq f(v_1) \geq f(v_2) \geq f(x)$ for all $x \in [a_{v_2},b_{v_2}]$ by definition of $\mathcal{V}$.  Thus, if $f(b_{v_2}) \leq f(x)$ for all $x \in [b_{v_2},c_1]$, then $\langle a_{v_2},b_{v_2},c_1 \rangle$ removes $v_2$.  Furthermore, if $\langle a_{v_2},b_{v_2},c_1 \rangle$ removes $v_2$, then $c_1$ is a target for $v_2$.

For (1), suppose $v_1 < v_2$.  Let $c$ be any target for $v_2$.  If $c \leq c_1$, then let $c_2 = c$.  If $c > c_1$, then $f(b_{v_2}) \leq f(x)$ for all $x \in [b_{v_2},c] \supset [b_{v_2},c_1]$, hence $\langle a_{v_2},b_{v_2},c_1 \rangle$ removes $v_2$, so we may let $c_2 = c_1$.  In either case, we have a target $c_2$ for $v_2$ with $c_2 \leq c_1$.

For (2), suppose $v_2 < v_1$.  Let $c$ be any target for $v_2$.  If $c < a_{v_1}$, then let $c_2 = c$.  If $c \geq a_{v_1}$, then
\begin{itemize}
\item $f(b_{v_2}) \leq f(x)$ for all $x \in [b_{v_2},c] \supseteq [b_{v_2},a_{v_1}]$,
\item $f(b_{v_2}) \leq f(a_{v_1}) \leq f(x)$ for all $x \in [a_{v_1},b_{v_1}]$, and
\item $f(b_{v_2}) \leq f(a_{v_1}) \leq f(b_{v_1}) \leq f(x)$ for all $x \in [b_{v_1},c_1]$,
\end{itemize}
hence $f(b_{v_2}) \leq f(x)$ for all $x \in [b_{v_2},c_1]$.  Thus, $\langle a_{v_2},b_{v_2},c_1 \rangle$ removes $v_2$, so $c_1$ is a target for $v_2$.
\end{proof}

\begin{claim}
\label{claim:choose targets}
We may choose, for each $v \in \mathcal{V}$, a target $c_v$ for $v$, such that for all $u,v \in \mathcal{V}$, if $u < v$, then $c_u < a_v$ or $c_v \leq c_u$.
\end{claim}

\begin{proof}[Proof of Claim~\ref{claim:choose targets}]
\renewcommand{\qedsymbol}{\textsquare (Claim~\ref{claim:choose targets})}
Fix an enumeration $\mathcal{V} = \{w_1,\ldots,w_m\}$ such that $f(w_1) \geq f(w_2) \geq \cdots \geq f(w_m)$.  Let $c_{w_1}$ be an arbitrary target for $w_1$.  To choose the targets for $w_2,\ldots,w_m$, we proceed recursively.

Let $k \in \{2,\ldots,m\}$ and suppose for induction that targets $c_{w_1},\ldots,c_{w_{k-1}}$ have been chosen for $w_1,\ldots,w_{k-1}$ such that for all $u,v \in \{w_1,\ldots,w_{k-1}\}$, if $u < v$, then $c_u < a_v$ or $c_v \leq c_u$.  We will choose $c_{w_k}$ below to ensure that
\begin{enumerate}[label=($\ast$)]
\item \label{rec} for all $u,v \in \{w_1,\ldots,w_k\}$, if $u < v$, then $c_u < a_v$ or $c_v \leq c_u$.
\end{enumerate}
To verify the condition \ref{rec}, we need only consider the cases where either $u = w_k$ or $v = w_k$.  In fact, it suffices to prove that:
\begin{enumerate}[label=(\alph{*})]
\item If there exists $u \in \{w_1,\ldots,w_{k-1}\}$ with $u < w_k$, and if $u_0$ is the largest such element, then $c_{w_k} \leq c_{u_0}$; and
\item If there exists $v \in \{w_1,\ldots,w_{k-1}\}$ with $w_k < v$, and if $v_0$ is the smallest such element, then either $c_{w_k} < a_{v_0}$ or $c_{v_0} = c_{w_k}$.
\end{enumerate}
Indeed, if $u \in \{w_1,\ldots,w_{k-1}\}$ and $u < w_k$, and if $u_0$ is the largest such element, then $u \leq u_0$ and $c_u > w_k > a_{u_0}$ (by definition of a target, since $f(u) \geq f(w_k)$), so by (a) and by induction, $c_{w_k} \leq c_{u_0} \leq c_u$.  Also, if $v \in \{w_1,\ldots,w_{k-1}\}$ and $w_k < v$, and if $v_0$ is the smallest such element, then by (b) either $c_{w_k} < a_{v_0} \leq a_v$ or $c_{v_0} = c_{w_k}$, and in the latter case, by induction, either $c_{w_k} = c_{v_0} < a_v$ or $c_v \leq c_{v_0} = c_{w_k}$.

To choose $c_{w_k}$, we consider two cases.

\medskip
\textbf{Case 1:} Suppose there exists $u \in \{w_1,\ldots,w_{k-1}\}$ with $u < w_k$, let $u_0$ be the largest such element, and suppose that for all $v \in \{w_1,\ldots,w_{k-1}\}$ with $w_k < v$ we have $c_{u_0} < a_v$.

By Claim~\ref{claim:target order}(1), we may choose $c_{w_k}$ such that $c_{w_k} \leq c_{u_0}$.  Here we immediately have condition (a) satisfied.  For (b), if there exists $v \in \{w_1,\ldots,w_{k-1}\}$ with $w_k < v$, and if $v_0$ is the smallest such element, then by assumption in this case, we have $c_{u_0} < a_{v_0}$.

\medskip
If Case 1 does not hold, then either (i) there does not exist $u \in \{w_1,\ldots,w_{k-1}\}$ with $u < w_k$, in which case there exists $v \in \{w_1,\ldots,w_{k-1}\}$ with $w_k < v$ since $k \geq 2$, or (ii) there exists $u \in \{w_1,\ldots,w_{k-1}\}$ with $u < w_k$, and if $u_0$ is the largest such element, then there exists $v \in \{1,\ldots,w_{k-1}\}$ with $w_k < v$ and $c_{u_0} \geq a_v$.  In the event that (ii) holds, by induction, $c_v \leq c_{u_0}$ for any such $v$.  In particular, if $v_0 \in \{w_1,\ldots,w_{k-1}\}$ is the smallest element with $w_k < v_0$, then $c_{v_0} \leq c_{u_0}$.  Also, if $u \in \{w_1,\ldots,w_{k-1}\}$ is any element with $u < w_k$, then by definition of a target, $c_u > w_k > a_{u_0}$, so $c_{u_0} \leq c_u$ by induction.  Thus, $c_{v_0} \leq c_{u_0} \leq c_u$.  The following Case 2 encapsulates both of these possibilities (i) and (ii).

\medskip
\textbf{Case 2:} Suppose that there exists $v \in \{w_1,\ldots,w_{k-1}\}$ with $w_k < v$, let $v_0$ be the smallest such element, and suppose that for all $u \in \{w_1,\ldots,w_{k-1}\}$ with $u < w_k$ we have $c_{v_0} \leq c_u$.

By Claim~\ref{claim:target order}(2), we may choose $c_{w_k}$ such that either $c_{w_k} < a_{v_0}$, or $c_{w_k} = c_{v_0}$.  In particular, note that $c_{w_k} \leq c_{v_0}$.  Here we immediately have condition (b) satisfied.  For (a), if there exists $u \in \{w_1,\ldots,w_{k-1}\}$ with $u < w_k$, and if $u_0$ is the largest such element, then $c_{w_k} \leq c_{v_0} \leq c_{u_0}$ by assumption in this case.

\medskip
This completes the recursive construction.
\end{proof}

\medskip
\noindent \textbf{Step 2.}
From the first part of the proof, we have a finite collection $\mathcal{V}$ of visors under $f$ such that (1) every visor under $f$ is contained in $\bigcup_{v \in \mathcal{V}} [a_v,b_v]$, (2) the intervals $[a_v,b_v]$, $v \in V$, are all distinct, and (3) for each $v \in \mathcal{V}$, $f(v) = \max \{f(x)$: $x \in [a_v,b_v]\}$.  Furthermore, we have, for each $v \in \mathcal{V}$, a target $c_v$ for $v$, such that for all $u,v \in \mathcal{V}$, if $u < v$, then $c_u < a_v$ or $c_v \leq c_u$.  See the top picture in Figure~\ref{fig:visor tuck} for an illustration of these elements in a specific example.

\begin{figure}
\begin{center}
\includegraphics{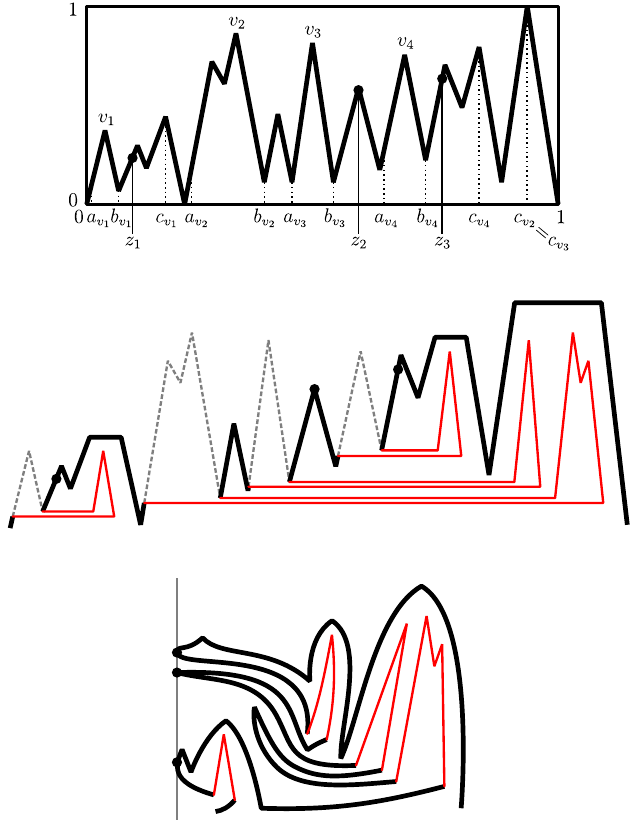}
\end{center}

\caption{First picture: An example of a map $f \colon I \to I$ and a set $Z = \{z_1,z_2,z_3\}$.  In this example, $\mathcal{V} = \{v_1,v_2,v_3,v_4\}$, where $v_1$ is a visor of $z_1$, $v_2$ and $v_3$ are visors of $z_2$, and $v_4$ is a visor of $z_3$.  For each $v \in \mathcal{V}$ the elements $a_v,b_v,c_v$ are displayed.  Second picture: An illustration of the image of $\Psi$ for this example.  The dashed lines show parts of the graph of $f$ which have been ``moved'' to the right (replaced by the parts shown in red).  Third picture: An illustration of the image of $\Phi$ for this example.  The points marked with dots are the images of $z_1,z_2,z_3$.}
\label{fig:visor tuck}
\end{figure}

We now proceed with constructing the desired embedding $\Phi$ of $I$ into $\mathbb{H}$.  By a small perturbation of $f$, we may construct a map $f' \colon I \to I$ with the following properties:
\begin{enumerate}
\item $\|\langle \frac{\varepsilon}{2},f'(x) \rangle - \langle 0,f(x) \rangle\| < \varepsilon$ for all $x \in I$;
\item For any $v \in \mathcal{V}$, $f'(a_v) < f'(x)$ for all $x \in (a_v,b_v]$ (if $a_v \neq 0$, or if $a_v = 0$ and $0 \in Z$), $f'(b_v) < f'(x)$ for all $x \in (b_v,c_v]$, and $f'(x) < f'(c_v)$ for all $x \in [a_v,b_v]$;
\item $f'(z) = f(z)$ for all $z \in Z$; and
\item For each $j \in \{1,\ldots,n\}$ and any $x \in [z_{j-1},z_j)$ (or $[0,z_1)$ in the case $j=1$), if $x \notin (a_v,b_v)$ for all $v \in \mathcal{V}$, then $f'(x) < f'(z_j)$.
\end{enumerate}

Let $\psi \colon [0,\frac{\varepsilon}{2}] \to I$ be a non-decreasing, onto map such that $\psi^{-1}(x)$ is a non-degenerate interval for each $x$ which is equal to either $a_v$ (if $a_v \neq 0$, or if $a_v = 0$ and $0 \in Z$), $b_v$, or $c_v$ for some $v \in \mathcal{V}$, and $\psi^{-1}(x)$ is a singleton for all other $x \in I$.  For each $v \in \mathcal{V}$, choose an interval $[b_v',a_v'] \subset \psi^{-1}(c_v)$, in such a way that for any $v_1,v_2 \in \mathcal{V}$ with $v_1 < v_2$, if $c_{v_1} = c_{v_2}$ then $a_{v_2}' < b_{v_1}'$.  Define $\Psi \colon [0,\frac{\varepsilon}{2}] \to \mathbb{H}$ as follows:
\begin{itemize}
\item For all $t \in [0,\frac{\varepsilon}{2}]$ such that $\psi(t) \notin \bigcup_{v \in \mathcal{V}} [a_v,b_v]$, define $\Psi(t) = \langle t,f'(\psi(t)) \rangle$;
\item For each $v \in \mathcal{V}$, if we denote $\psi^{-1}(a_v) = [t_1,t_2]$ (or $\psi^{-1}(a_v) = \{t_2\}$ if $a_v = 0$) and $\psi^{-1}(b_v) = [u_1,u_2]$, then:
\begin{itemize}
\item If $a_v \neq 0$, or if $a_v = 0$ and $0 \in Z$, $\Psi {\restriction}_{[t_1,t_2]}$ parameterizes the horizontal segment from $\langle t_1,f'(a_v) \rangle$ to $\langle a_v',f'(a_v) \rangle$;
\item $\Psi {\restriction}_{[u_1,u_2]}$ parameterizes the horizontal segment from $\langle b_v',f'(b_v) \rangle$ to $\langle u_2,f'(b_v) \rangle$; and
\item For each $t \in [t_2,u_1]$, define $\Psi(t) = \langle \rho(t),f'(\psi(t)) \rangle$, where $\rho \colon [t_2,u_1] \to [b_v',a_v']$ is the linear homeomorphism with $\rho(t_2) = a_v'$ and $\rho(u_1) = b_v'$.
\end{itemize}
\end{itemize}
See the second picture in Figure~\ref{fig:visor tuck} for an illustration of the image of this map $\Psi$ in a specific example.

Note that $\Psi$ is one-to-one, so $\Psi([0,\frac{\varepsilon}{2}])$ is an arc in $\mathbb{H}$.  Moreover, by construction, for each $z \in Z$, the intersection of $\Psi([0,\frac{\varepsilon}{2}])$ with the horizontal arc $[0,\min \psi^{-1}(z)] \times \{f(z)\}$ is $\{\langle \min \psi^{-1}(z),f(z) \rangle\}$.

Let $M \colon \mathbb{H} \to \mathbb{H}$ be a monotone map such that the $y$-coordinate of $M(p)$ equals the $y$-coordinate of $p$ for each $p \in \mathbb{H}$, $M([0,\frac{\varepsilon}{2}] \times \mathbb{R}) \subseteq [0,\frac{\varepsilon}{2}] \times \mathbb{R}$, $M(\Psi(\psi^{-1}(x)))$ is a singleton for each $x \in I$, $M([0,\min \psi^{-1}(z)] \times \{f(z)\}) = \{\langle 0,f(z) \rangle\}$ for each $z \in Z$, and $M$ is otherwise one-to-one.  Finally, define $\Phi \colon I \to M(\Psi([0,\frac{\varepsilon}{2}]))$ by $\Phi(x) = M(\Psi(\psi^{-1}(x)))$ (see the third picture in Figure~\ref{fig:visor tuck} for an illustration).  This $\Phi$ has the desired properties.
\end{proof}

The next two results provide sufficient conditions under which an arc-like continuum may be embedded in the plane so as to make a given finite (Proposition~\ref{prop:embed n pts accessible}) or countably infinite (Proposition~\ref{prop:embed ctble set accessible}) set accessible (compare with Proposition~5.2 of \cite{ammerlaan-anusic-hoehn2023}).  Both can be proved using Lemma~\ref{lem:tuck} above, though the proof of Proposition~\ref{prop:embed ctble set accessible} involves somewhat more extensive bookkeeping.  We present the proof of Proposition~\ref{prop:embed ctble set accessible} only; it may easily be adapted to establish Proposition~\ref{prop:embed n pts accessible} as well.

Given an inverse limit $X = \varprojlim \left\langle I,f_i \right\rangle$, for each $i \in \mathbb{N}$, let $\pi_i \colon X \to I$ be the $i$-th coordinate projection, $\pi_i(\langle x_1,x_2,\ldots, \rangle) = x_i$.

\begin{prop}
\label{prop:embed n pts accessible}
Let $X = \varprojlim \left\langle I,f_i \right\rangle$ be an arc-like continuum, and let $\mathbf{Z} = \{\mathbf{z}_1,\ldots,\mathbf{z}_n\}$ be a set of $n$ points in $X$.  Suppose that for each $i \in \mathbb{N}$,
\begin{enumerate}
\item The points $\pi_{i+1}(\mathbf{z}_1),\ldots,\pi_{i+1}(\mathbf{z}_n)$ are all distinct;
\item For each $z,z' \in \pi_{i+1}(\mathbf{Z})$, if $z < z'$ then $f_i(z) < f_i(z')$; and
\item Every $\pi_{i+1}(\mathbf{Z})$-visor under $f_i$ is removable.
\end{enumerate}
Then there exists an embedding $\Omega \colon X \to \mathbb{R}^2$ such that $\Omega(\mathbf{z})$ is an accessible point of $\Omega(X)$ for each $\mathbf{z} \in \mathbf{Z}$.
\end{prop}

\begin{prop}
\label{prop:embed ctble set accessible}
Let $X = \varprojlim \left\langle I,f_i \right\rangle$ be an arc-like continuum, and let $\mathbf{Z} = \{\mathbf{z}_1,\mathbf{z}_2,\ldots\}$ be a countably infinite set of points in $X$.  Suppose that for each $i \in \mathbb{N}$,
\begin{enumerate}
\item The points $\pi_{i+1}(\mathbf{z}_1),\ldots,\pi_{i+1}(\mathbf{z}_i)$ are all distinct;
\item For each $z,z' \in \pi_{i+1}(\{\mathbf{z}_1,\ldots,\mathbf{z}_i\})$, if $z < z'$ then $f_i(z) < f_i(z')$; and
\item Every $\pi_{i+1}(\{\mathbf{z}_1,\ldots,\mathbf{z}_i\})$-visor under $f_i$ is removable.
\end{enumerate}
Then there exists an embedding $\Omega \colon X \to \mathbb{R}^2$ such that $\Omega(\mathbf{z})$ is an accessible point of $\Omega(X)$ for each $\mathbf{z} \in \mathbf{Z}$.
\end{prop}

\begin{proof}
Consider the space $X^+$ obtained by adding a null sequence of pairwise disjoint arcs $A_1,A_2,\ldots$ to $X$, where for each $i = 1,2,\ldots$, $A_i$ is attached at an endpoint to $\mathbf{z}_i$, and is otherwise disjoint from $X$.  We may express $X^+$ as an inverse limit as follows.  For each $i = 1,2,\ldots$, let $T_i \subset \mathbb{R}^2$ be the tree which is the union of the arc $\{0\} \times I$ with the arcs $[-1,0] \times \pi_i(\mathbf{z}_j)$, $j = 1,\ldots,i-1$.  Define $f_i^+ \colon T_{i+1} \to T_i$ by $f_i^+(\langle 0,y \rangle) = \langle 0,f_i(y) \rangle$ for each $y \in I$, $f_i^+(\langle x,\pi_{i+1}(\mathbf{z}_j) \rangle) = \langle x,\pi_i(\mathbf{z}_j) \rangle$ for each $x \in [-1,0]$ and each $j = 1,\ldots,i-1$, and $f_i^+(\langle x,\pi_{i+1}(\mathbf{z}_i) \rangle) = \langle 0,\pi_i(\mathbf{z}_i) \rangle$ for each $x \in [-1,0]$.  Then $X^+ \approx \varprojlim \left\langle T_i,f_i^+ \right\rangle$.

We apply the Anderson-Choquet Embedding Theorem \cite{anderson-choquet1959} to show that $X^+$ can be embedded in $\mathbb{R}^2$.  We will construct a sequence of orientation-preserving homeomorphisms $\Theta_i \colon \mathbb{R}^2 \to \mathbb{R}^2$, $i = 1,2,\ldots$, in such a way that $X^+$ is homeomorphic to the Hausdorff limit of the sequence of embeddings $\Theta_i(T_i)$ of the factor spaces $T_i$.  It suffices to prove that for any $i \in \mathbb{N}$, any orientation-preserving homeomorphism $\Theta \colon \mathbb{R}^2 \to \mathbb{R}^2$, and any $\varepsilon > 0$, there exists an orientation-preserving homeomorphism $\Theta' \colon \mathbb{R}^2 \to \mathbb{R}^2$ such that for each $p \in \Theta'(T_{i+1})$, $\|p - \Theta \circ f_i^+ \circ (\Theta')^{-1}(p)\| < \varepsilon$.

Let $i \in \mathbb{N}$, let $\Theta \colon \mathbb{R}^2 \to \mathbb{R}^2$ be an orientation-preserving homeomorphism, and let $\varepsilon > 0$.  Let $\varepsilon' > 0$ be small enough so that for any $p,q \in \mathbb{R}^2$ with $p \in T_i$ and $\|p - q\| \leq \varepsilon'$, we have $\|\Theta(p) - \Theta(q)\| < \varepsilon$.

By Lemma~\ref{lem:tuck}, there exists an embedding $\Phi$ of $I$ in $\mathbb{R}^2$ such that
\begin{enumerate}
\item $\|\Phi(x) - \langle 0,f_i(x) \rangle\| < \varepsilon'$ for all $x \in I$;
\item $\Phi(\pi_{i+1}(\mathbf{z}_j)) = \langle 0,\pi_i(\mathbf{z}_j) \rangle$ for each $j \in \{1,\ldots,i\}$, and $\Phi(x) \notin \partial \mathbb{H}$ for all $x \in I \smallsetminus \pi_{i+1}(\{\mathbf{z}_1,\ldots,\mathbf{z}_i\})$; and
\item For any $j,j' \in \{1,\ldots,i\}$ such that $\pi_{i+1}(\mathbf{z}_j), \pi_{i+1}(\mathbf{z}_{j'})$ are consecutive (in the sense of the ordering in $I$) points in $\pi_{i+1}(\{\mathbf{z}_1,\ldots,\mathbf{z}_i\})$, the set
\[ \Phi \left( I \smallsetminus [\pi_{i+1}(\mathbf{z}_j),\pi_{i+1}(\mathbf{z}_{j'})] \right) \]
is disjoint from the bounded complementary component of the circle
\[ \Phi \left( [\pi_{i+1}(\mathbf{z}_j),\pi_{i+1}(\mathbf{z}_{j'})] \right) \cup \left( \{0\} \times [\pi_i(\mathbf{z}_j),\pi_i(\mathbf{z}_{j'})] \right) .\]
\end{enumerate}

Define the embedding $H \colon T_{i+1} \to \mathbb{R}^2$ by $H(\langle 0,y \rangle) = \Phi(y)$ for each $y \in I$, $H(\langle x,\pi_{i+1}(\mathbf{z}_j) \rangle) = \langle x,\pi_i(\mathbf{z}_j) \rangle$ for each $j \in \{1,\ldots,i-1\}$ and $x \in [-1,0]$, and $H(\langle x,\pi_{i+1}(\mathbf{z}_i) \rangle) = \langle x \cdot \varepsilon',\pi_i(\mathbf{z}_i) \rangle$ for each $x \in [-1,0]$.  It follows from property (3) above for $\Phi$ that $H$ preserves the local orientation in the plane at each point $\langle 0,\pi_{i+1}(\mathbf{z}_j) \rangle \in T_{i+1}$, $j \in \{1,\ldots,i\}$.  Therefore, $H$ can be extended to an orientation-preserving homeomorphism of the plane $\widehat{H} \colon \mathbb{R}^2 \to \mathbb{R}^2$ (see e.g.\ \cite{adkisson-maclane1940}).  Define the orientation-preserving plane homeomorphism $\Theta'$ by $\Theta' = \Theta \circ \widehat{H}$.

Let $p \in \Theta'(T_{i+1})$.  To prove that $\|p - \Theta \circ f_i^+ \circ (\Theta')^{-1}(p)\| < \varepsilon$, we consider three cases for the point $(\Theta')^{-1}(p) \in T_{i+1}$.

\medskip
\textbf{Case 1:} $(\Theta')^{-1}(p)$ has the form $\langle 0,y \rangle$ for some $y \in I$.  Then by property (1) above for $\Phi$, and the definitions of $f_i^+$ and $H$, $\|H((\Theta')^{-1}(p)) - f_i^+((\Theta')^{-1}(p))\| < \varepsilon'$.  Therefore, by choice of $\varepsilon'$, we have $\|\Theta(H((\Theta')^{-1}(p))) - \Theta(f_i^+((\Theta')^{-1}(p))) \| < \varepsilon$, i.e.\ $\|p - \Theta \circ f_i^+ \circ (\Theta')^{-1}(p)\| < \varepsilon$, as desired.

\medskip
\textbf{Case 2:} $(\Theta')^{-1}(p)$ has the form $\langle x,\pi_{i+1}(\mathbf{z}_j) \rangle$ for some $j \in \{1,\ldots,i-1\}$ and $x \in [-1,0]$.  In this case, $H((\Theta')^{-1}(p)) = f_i^+((\Theta')^{-1}(p)) = \langle x,\pi_i(\mathbf{z}_j) \rangle$, so $p = \Theta(H((\Theta')^{-1}(p))) = \Theta \circ f_i^+ \circ (\Theta')^{-1}(p)$.

\medskip
\textbf{Case 3:} $(\Theta')^{-1}(p)$ has the form $\langle x,\pi_{i+1}(\mathbf{z}_i) \rangle$ for some $x \in [-1,0]$.  In this case, $H((\Theta')^{-1}(p)) = \langle x \cdot \varepsilon',\pi_i(\mathbf{z}_i) \rangle$ and $f_i^+((\Theta')^{-1}(p)) = \langle 0,\pi_i(\mathbf{z}_i) \rangle$.  Therefore, by choice of $\varepsilon'$, we have $\|\Theta(H((\Theta')^{-1}(p))) - \Theta(f_i^+((\Theta')^{-1}(p))) \| < \varepsilon$, i.e.\ $\|p - \Theta \circ f_i^+ \circ (\Theta')^{-1}(p)\| < \varepsilon$, as desired.
\end{proof}

\section{Knaster continua}
\label{sec:knaster}

By a \emph{Knaster continuum}, we mean a continuum which is homeomorphic to an inverse limit $\varprojlim \left\langle I,f_i \right\rangle$, where each bonding map $f_i$ is an open map which is not a homeomorphism.  In fact, the maps $f_i$ may be taken to be tent maps $T_{m_i}$ \cite[Theorem 7]{rogers1970}; see Definition~\ref{def:Tm} below for the definition of $T_m$.  Each Knaster continuum is indecomposable, and has either one or two endpoints.

Note that for an open map $f \colon I \to I$, we have that $f(\{0,1\}) = \{0,1\}$, and a point $x \in I$ is an extremum of $f$ if and only if $f(x) \in \{0,1\}$.  Moreover, given a subinterval $A \subset I$, if the interior of $A$ contains no extrema of $f$, then $f {\restriction}_A$ is one-to-one, and if $A$ contains at least two extrema, then $f(A) = I$.

\begin{lem}
\label{lem:Knaster composants}
Let $K = \varprojlim \left\langle I,f_i \right\rangle$ be a Knaster continuum, where each $f_i$ is an open map, and let $\mathbf{x} = \langle x_i \rangle_{i=1}^\infty$ and $\mathbf{y} = \langle y_i \rangle_{i=1}^\infty$ be two points of $K$.  Then $\mathbf{x}$ and $\mathbf{y}$ belong to different composants of $K$ if and only if for infinitely many $i$ there is an extremum of $f_i$ between $x_{i+1}$ and $y_{i+1}$.
\end{lem}

\begin{proof}
For each $i \in \mathbb{N}$, let $A_i$ be the closed interval with endpoints $x_i$ and $y_i$.

First, suppose that there exists $N \in \mathbb{N}$ such that for all $i \geq N$, $A_{i+1}$ contains no extremum of $f_i$ in its interior.  Then for all $i \geq N$, $A_{i+1} \neq I$, and $f_i$ is one-to-one on $A_{i+1}$ with $f_i(A_{i+1}) = A_i$.  Hence, $\varprojlim \left\langle A_i,f_i \right\rangle_{i \geq N}$ is a proper subcontinuum (an arc) of $\varprojlim \left\langle I,f_i \right\rangle_{i \geq N} \approx K$ containing $\mathbf{x}$ and $\mathbf{y}$.  Thus, $\mathbf{x}$ and $\mathbf{y}$ are in the same composant of $K$.

Conversely, suppose that for infinitely many $i$ there is an extremum of $f_i$ between $x_{i+1}$ and $y_{i+1}$.  Let $Y$ be a subcontinuum of $K$ containing $\mathbf{x}$ and $\mathbf{y}$.  Let $i_0 \in \mathbb{N}$ be arbitrary.  Choose $i_0 \leq i_1 < i_2$ such that there is an extremum of $f_{i_1}$ between $x_{i_1+1}$ and $y_{i_1+1}$ and an extremum of $f_{i_2}$ between $x_{i_2+1}$ and $y_{i_2+1}$.  Since $A_{i_2+1}$ contains an extremum of $f_{i_2}$, we have that $f_{i_2}(A_{i_2+1})$ contains $0$ or $1$.  In fact, since $f_i(\{0,1\}) \subseteq \{0,1\}$ for all $i$, it follows that $f_i^{i_2+1}(A_{i_2+1})$ contains $0$ or $1$ for each $i \leq i_2$.  In particular, $f_{i_1+1}^{i_2+1}(A_{i_2+1})$ contains $0$ or $1$.  Additionally, $f_{i_1+1}^{i_2+1}(A_{i_2+1})$ contains $x_{i_1+1}$ and $y_{i_1+1}$, hence it contains the extrema between them.  We deduce that $f_{i_1}^{i_2+1}(A_{i_2+1}) = I$.  Therefore, $\pi_{i_0}(Y) \supseteq f_{i_0}^{i_2+1}(A_{i_2+1}) = I$.  Since $i_0$ was arbitrary, we conclude that $\pi_i(Y) = I$ for all $i \in \mathbb{N}$.  Thus, $Y = K$, and so $\mathbf{x}$ and $\mathbf{y}$ are in the different composants of $K$.
\end{proof}

\begin{cor}
\label{cor:comp extrema}
Let $K = \varprojlim \left\langle I,f_i \right\rangle$ be a Knaster continuum, where each $f_i$ is an open map, and let $\mathbf{x} = \langle x_i \rangle_{i=1}^\infty$ and $\mathbf{y} = \langle y_i \rangle_{i=1}^\infty$ be two points of $K$ which lie in different composants of $K$.  Then for each $N$ and $i$, there exists $k_0 > i$ such that for all $k \geq k_0$, there are at least $N$ extrema between $x_k$ and $y_k$ in $f_i^k$.
\end{cor}

\begin{proof}
Choose $i_0 > i$ large enough so that $f_i^{i_0}$ has at least $N$ extrema.  As in the proof of Lemma~\ref{lem:Knaster composants}, there exists $k_0 \geq i_0$ (there $k_0 = i_2+1$) such that $f_{i_0}^{k_0}(A_{k_0}) = I$, where $A_{k_0}$ is the closed interval with endpoints $x_{k_0}$ and $y_{k_0}$.  It follows easily that for all $k \geq k_0$, there are at least $N$ extrema between $x_k$ and $y_k$ in $f_i^k = f_i^{i_0} \circ f_{i_0}^k$.
\end{proof}

\section{Factorization of tent maps}
\label{sec:tent factorization}

In this Section we collect some technical results about some interval maps, specifically tent maps, which will be used as ingredients in the proofs of the main results in the next Section.  The key idea here is encapsulated in Definition~\ref{def:Tm factorization}, which provides a factorization of a tent map (see Lemma~\ref{lem:Tm factorization}) which is useful for constructing alternative inverse limit representations of Knaster continua in which certain visors are removable.

\begin{defn}
Let $f \colon I \to I$ be a map.  An interval $P = [a,b] \subseteq I$ is called a \emph{sawtooth pattern of $f$} if there are $h \in (0,1]$ and $m \geq 1$ such that
\begin{itemize}
\item $f \left( a + j \cdot \frac{b-a}{m} \right) = 0$ for all even $j \in \{0,\ldots,m\}$,
\item $f \left( a + j \cdot \frac{b-a}{m} \right) = h$ for all odd $j \in \{0,\ldots,m\}$, and
\item $f$ is linear on $\left[ a + j \cdot \frac{b-a}{m}, a + (j+1) \cdot \frac{b-a}{m} \right]$ for each $j \in \{0,\ldots,m-1\}$.
\end{itemize}
We say that the sawtooth pattern $P$ has $\frac{m}{2}$ \emph{teeth} and \emph{height} $h$.  A sawtooth pattern with one tooth will be called simply a \emph{tooth}.  If $P = [a,b] \subset I$ is a sawtooth pattern of $f$ with $\frac{m}{2}$ teeth, and if $j \in \left\{ 1,\ldots,\lfloor \frac{m}{2} \rfloor \right\}$, then the interval $[a+(2j-2)\cdot\frac{b-a}{m}, a+2j\cdot\frac{b-a}{m}]$ is called the \emph{$j$-th tooth of $P$}.  In the case that $m$ is odd, then the interval $[b - \frac{b-a}{m}, b]$ is called the \emph{half-tooth} of $P$.
\end{defn}

\begin{defn}
\label{def:Tm}
Let $m \geq 1$.  The \emph{$m$-tent map} is the map $T_m \colon I \to I$ such that
\begin{itemize}
\item $T_m \left( \frac{j}{m} \right) = 0$ for all even $j \in \{0,\ldots,m\}$,
\item $T_m \left( \frac{j}{m} \right) = 1$ for all odd $j \in \{0,\ldots,m\}$, and
\item $T_m$ is linear on $\left[ \frac{j}{m}, \frac{j+1}{m} \right]$ for each $j \in \{0,\ldots,m-1\}$.
\end{itemize}
Equivalently, $T_m$ is characterized by the property that $I$ itself is a sawtooth pattern of $T_m$ with $\frac{m}{2}$ teeth of height $1$.
\end{defn}

We point out that for any $m_1,m_2 \geq 1$, $T_{m_1} \circ T_{m_2} = T_{m_1 \cdot m_2}$.  In fact, more generally, we have the following simple result, which we leave to the reader.

\begin{lem}
\label{lem:tent sawtooth comp}
Let $m,m' \geq 1$.  Let $f \colon I \to I$ be a map, and suppose $P \subseteq I$ is a sawtooth pattern for $f$ with $\frac{m'}{2}$ teeth and height $\frac{i}{m}$, for some $i \in \{1,\ldots,m\}$.  Then $P$ is a sawtooth pattern for $T_m \circ f$ with $\frac{m' \cdot i}{2}$ teeth and height $1$.

In particular, if $P$ is a tooth for $f$ of height $\frac{i}{m}$, then $P$ is a sawtooth pattern for $T_m \circ f$ with $i$ teeth and height $1$.
\end{lem}

\begin{defn}
\label{def:Tm factorization}
Let $m \geq 1$, and let $Z = \{z_1,\ldots,z_n\} \subset I$ with $z_1 < \cdots < z_n$.  Assume that there exist intervals $P_1,\ldots,P_n \subset I$ with mutually disjoint interiors such that:
\begin{enumerate}
\item for every $i \in \{1,\ldots,n-1\}$, $P_i$ is a sawtooth pattern for $T_m$ with $2i-1$ teeth (and height $1$), and $z_i$ is contained in the first half (the increasing part) of the $i$-th tooth of $P_i$; and
\item either (a) $P_n$ is either a sawtooth pattern for $T_m$ with $2n-1$ teeth and $z_n$ is contained in the first half (the increasing part) of the $n$-th tooth of $P_n$, or (b) $1 \in P_n$, and $P_n$ is a sawtooth pattern for $T_m$ with $n - \frac{1}{2}$ teeth, and $z_n$ is in the half-tooth of $P_n$.
\end{enumerate}
Then we define $s = s_{m,Z} \colon I \to I$ as follows:
\begin{itemize}
\item For $x \in I \smallsetminus \bigcup_{i=1}^n P_i$, let $s(x) = \frac{T_m(x)}{2n-1}$.
\item For $i \in \{1,\ldots,n-1\}$, $s {\restriction}_{P_i}$ is a tooth of height $\frac{2i-1}{2n-1}$.
\item In case (a) above, we let $s {\restriction}_{P_n}$ be a tooth of height $1$.  In case (b), we let $s {\restriction}_{P_n}$ be a linear, increasing map such that $s(P_n) = I$.
\end{itemize}
\end{defn}

See Figure~\ref{fig:knaster ex} below for an example of this map $s_{m,Z}$, for $m = 16$ and a specific set $Z$.

\begin{lem}
\label{lem:Tm factorization}
Let $m$, $Z = \{z_1,\ldots,z_n\}$, and $s$ be as in Definition~\ref{def:Tm factorization}.  Then
\[ T_{2n-1} \circ s = T_m .\]
\end{lem}

\begin{proof}
Recall that $T_{2n-1}(x) = (2n-1)x$ for every $x \in [0,\frac{1}{2n-1}]$.  So, if $x \in I \smallsetminus \bigcup_{i=1}^n P_i$, then $s(x) = \frac{T_m(x)}{2n-1} \in [0,\frac{1}{2n-1}]$, which implies $T_{2n-1} \circ s(x) = T_m(x)$.

For $P_i$ such that $s {\restriction}_{P_i}$ is a tooth of height $\frac{2i-1}{2n-1}$, by Lemma~\ref{lem:tent sawtooth comp}, $P_i$ is a sawtooth pattern of $T_{2n-1} \circ s$ with $2i-1$ teeth and height $1$.  In other words, $T_{2n-1} \circ s$ equals $T_m$ on $P_i$.

If $1 \in P_n$, and $P_n$ is a sawtooth pattern with $n - \frac{1}{2}$ teeth, then $s {\restriction}_{P_n}$ is linear, increasing, and $s(P_n) = I$.  Then $T_{2n-1} \circ s {\restriction}_{P_n}$ is obviously a sawtooth pattern with $n - \frac{1}{2}$ teeth, as is $T_m {\restriction}_{P_n}$.  So, $T_{2n-1} \circ s {\restriction}_{P_n} = T_m {\restriction}_{P_n}$.
\end{proof}

\begin{lem}
\label{lem:s images}
Let $m$, $Z = \{z_1,\ldots,z_n\}$, and $s$ be as in Definition~\ref{def:Tm factorization}.  For each $i \in \{1,\ldots,n\}$,
\[ s(z_i) \in \left[ \frac{2i-2}{2n-1}, \frac{2i-1}{2n-1} \right] .\]
In particular, the map $s$ is order-preserving on $Z$, i.e.\ $s(z_1) < \cdots < s(z_n)$.

Moreover, $s$ is increasing at $z_i$, and $T_{2n-1}$ is increasing at $s(z_i)$, for each $i \in \{1,\ldots,n\}$.
\end{lem}

\begin{proof}
Let $i \in \{1,\ldots,n\}$, and let $P_i = [a_i,b_i]$.

Suppose first that $P_i$ is a sawtooth pattern for $T_m$ with $2i-1$ teeth, i.e.\ that either $i < n$, or $i = n$ and we are in case (a) of condition (2) in Definition \ref{def:Tm factorization}.  On $P_i$, $s$ is defined to be a tooth of height $\frac{2i-1}{2n-1}$.  Thus, on the first half of $P_i$, $\left[ a_i, \frac{a_i+b_i}{2} \right]$, we have $s(x) = \frac{2i-1}{2n-1} \cdot \frac{2}{b_i-a_i} \cdot (x - a_i)$.  The assumption about the location of $z_i$ within $P_i$ tells us that $z_i \in \left[ a_i + \frac{b_i - a_i}{2} \cdot \frac{2i-2}{2i-1}, a_i + \frac{b_i - a_i}{2} \right]$.

Now suppose instead that $i = n$, $b_n = 1$, and $P_n$ is a sawtooth pattern for $T_m$ with $n - \frac{1}{2}$ teeth, i.e.\ that we are in case (b) of condition (2) in Definition \ref{def:Tm factorization}.  In this case, on $P_n$, $s$ is defined to be a linear, increasing map such that $s(P_n) = I$.  That is, $s(x) = \frac{1}{1-a_n} \cdot (x - a_n)$.  The assumption about the location of $z_n$ within $P_n$ tells us that $z_n \in \left[ a_n + (1-a_n) \cdot \frac{2n-2}{2n-1}, 1 \right]$.

In either case, we obtain $s(z_i) \in \left[ \frac{2i-2}{2n-1}, \frac{2i-1}{2n-1} \right]$.  It follows immediately that $s(z_1) < \cdots < s(z_n)$.

Moreover, observe that $T_{2n-1}$ is increasing on $\left[ \frac{2i-2}{2n-1}, \frac{2i-1}{2n-1} \right]$.  So, $T_m$ is increasing at $z_i$, and $T_{2n-1}$ is increasing at $s(z_i)$.  According to Lemma~\ref{lem:Tm factorization}, $T_m = T_{2n-1} \circ s$.  Therefore, we must have that $s$ is increasing at $z_i$ as well.
\end{proof}

For the next Lemma, the reader may find it helpful to refer to refer to Figure~\ref{fig:knaster ex} below for an example featuring the key elements of this result.

\begin{lem}
\label{lem:s Tell visors removable}
Let $m$, $Z = \{z_1,\ldots,z_n\}$, and $s$ be as in Definition~\ref{def:Tm factorization}.  Let $\ell \geq 1$, and let $Z' = \{z_1',\ldots,z_n'\} \subset I$ with $z_1' < \cdots < z_n'$.  Assume that $T_\ell(z_i') = z_i$, and $T_\ell$ is increasing at $z_i'$, for every $i \in \{1,\ldots,n\}$.  Then every $Z'$-visor under $s \circ T_\ell$ is removable.
\end{lem}

\begin{proof}
By the definition of $s$, either each point of $I$ belongs to a tooth of $s$, or there exists $q \in I$ such that $s$ is increasing on $[q,1]$ and $s([q,1]) = I$, and each point of $[0,q]$ belongs to a tooth of $s$.  In particular, for each $i \in \{1,\ldots,n-1\}$, the point $z_i$ is in the first half of a tooth of $s$ of height $\frac{2i-1}{2n-1}$, and also $z_n$ is either in the first half of a tooth of height $1$, or, in the latter case above, may be in the interval $[q,1]$.

It follows that, similarly, either each point of $I$ belongs to a tooth of $s \circ T_\ell$, or there exists $q' \in I$ such that $s \circ T_\ell$ is increasing on $[q',1]$ and $s \circ T_\ell([q',1]) = I$, and each point of $[0,q']$ belongs to a tooth of $s \circ T_\ell$.  Moreover, for each $i \in \{1,\ldots,n-1\}$, since $T_\ell(z_i') = z_i$ and $T_\ell$ is increasing at $z_i'$, we have that $z_i'$ is in the first half of a tooth of $s \circ T_\ell$ of height $\frac{2i-1}{2n-1}$, and also $z_n'$ is either in the first half of a tooth of height $1$, or, in the latter case above, may be in the interval $[q',1]$.

Let $i \in \{1,\ldots,n-1\}$.  Since $z_i'$ is in the first half of a tooth of $s \circ T_\ell$, we know that no point in this tooth may be a visor for $z_i'$ under $s \circ T_\ell$.  Also, since this tooth has height $\frac{2i-1}{2n-1}$, and $s \circ T_\ell(z_{i+1}') = s(z_{i+1}) \in \left[ \frac{2i}{2n-1}, \frac{2i+1}{2n-1} \right]$ by Lemma~\ref{lem:s images}, we see that no point in this tooth may be a visor for $z_{i+1}'$ under $s \circ T_\ell$ either.  Thus, the points in the teeth containing the points $z_i'$, $i \in \{1,\ldots,n-1\}$, are not $Z'$-visors under $s \circ T_\ell$.  Similarly, no point in the increasing interval containing $z_n'$ can be a $Z'$-visor under $s \circ T_\ell$, and no point greater than $z_n'$ can be a $Z'$-visor under $s \circ T_\ell$ either.

It follows that if $v$ is a visor under $s \circ T_\ell$, then $v$ belongs to a tooth $[a,b]$ of $s \circ T_\ell$ which does not contain any of the points of $Z'$.  Therefore, $v$ is removable; specifically, $\langle a,b,1 \rangle$ removes $v$.
\end{proof}

\section{Plane embeddings of Knaster continua}
\label{sec:knaster embeddings}

\begin{thm}
\label{thm:knaster ctble composants}
Let $K$ be a Knaster continuum, and let $\{\kappa_1,\kappa_2,\ldots\}$ be any countable set of composants of $K$.  Then there exists an embedding $\Omega \colon K \to \mathbb{R}^2$ such that for every $n \in \mathbb{N}$, $\Omega(\kappa_n)$ is a fully accessible composant of $\Omega(K)$.
\end{thm}

\begin{proof}
We may assume that the set $\{\kappa_1,\kappa_2,\ldots\}$ is infinite, and that $\kappa_{n_1} \neq \kappa_{n_2}$ whenever $n_1 \neq n_2$.

Express $K$ as an inverse limit of tent maps; i.e.\ assume $K = \varprojlim \left\langle I,g_i \right\rangle$, where for each $i \in \mathbb{N}$, $g_i = T_{m_i}$ for some integer $m_i \geq 2$.  For each $i \in \mathbb{N}$, let $\pi_i \colon K = \varprojlim \left\langle I,g_i \right\rangle \to I$ be the $i$-th coordinate projection, $\pi_i(\langle x_1,x_2,\ldots \rangle) = x_i$.

For each $n \in \mathbb{N}$, let $\mathbf{z}_n \in \kappa_n$ be a point which is not an endpoint of $K$.  In light of \cite[Lemma A]{debski-tymchatyn1993}, to prove the Theorem, it suffices to produce an embedding $\Omega \colon K \to \mathbb{R}^2$ such that for every $n \in \mathbb{N}$, $\Omega(\mathbf{z}_n)$ is an accessible point of $\Omega(K)$.  We will apply Proposition~\ref{prop:embed ctble set accessible} to obtain such an embedding, but to do this we need to construct an alternative inverse limit representation of $K$, $K \approx \varprojlim \left\langle I,f_i \right\rangle$, satisfying the conditions of Proposition~\ref{prop:embed ctble set accessible}.  The remainder of this proof is devoted to a recursive construction, encapsulated in the following Claim, after which we will define the maps $f_i$, $i \in \mathbb{N}$.

\begin{claim}
\label{claim:subsequences}
There exist infinite sets $J_i = \{j_i(1),j_i(2),\ldots\} \subseteq \mathbb{N}$, $i \in \mathbb{N}$, enumerated in increasing order $j_i(1) < j_i(2) < \cdots$, such that for each $i \in \mathbb{N}$, letting $\mathbf{Z}_i = \{\mathbf{z}_1,\ldots,\mathbf{z}_i\}$, we have:
\begin{enumerate}
\item $J_{i+1} \subseteq J_i \smallsetminus \{j_i(1)\}$;
\item for all $j \in J_i$, $\pi_j$ is one-to-one on $\mathbf{Z}_i$;
\item for all $j,j' \in J_i$ with $j < j'$, $g_j^{j'}$ is order-preserving on the set $\pi_{j'}(\mathbf{Z}_i)$;
\item for all $j,j' \in J_i$ with $j < j'$, for each $z \in \pi_{j'}(\mathbf{Z}_i)$, $g_j^{j'}$ is increasing at $z$;
\item for all $j \in J_i \smallsetminus \{j_i(1)\}$, for each two distinct points $z,z' \in \pi_j(\mathbf{Z}_i)$, $g_{j_i(1)}^j$ has at least $4i$ extrema between $z$ and $z'$; and
\item for all $j \in J_i \smallsetminus \{j_i(1)\}$, for each $z \in \pi_j(\mathbf{Z}_i)$, either $g_{j_i(1)}^j$ has at least $2i$ extrema between $z$ and $1$, or $g_{j_i(1)}^j$ is increasing on $[z,1]$.
\end{enumerate}
\end{claim}

\begin{proof}[Proof of Claim~\ref{claim:subsequences}]
\renewcommand{\qedsymbol}{\textsquare (Claim~\ref{claim:subsequences})}
For the sake of the recursion, let $J_0 = \mathbb{N}$, with $j_0(1) = 1$.

Now let $i \in \mathbb{N}$, and suppose that $J_{i-1}$ has been defined so that properties (1)--(6) hold.  Choose
\begin{enumerate}[label=(\arabic{*}$'$), series=rec_prime]
\item $a_0 > j_{i-1}(1)$
\end{enumerate}
large enough so that $\pi_{a_0}$ is one-to-one on $\mathbf{Z}_i$.  It follows that
\begin{enumerate}[resume*=rec_prime]
\item $\pi_j$ is one-to-one on $\mathbf{Z}_i$ for all $j \geq a_0$.
\end{enumerate}
By the pigeonhole principle, there exists an infinite set $A_1 \subseteq J_{i-1} \smallsetminus \{j_{i-1}(1)\}$ such that for all $j \in A_1$, $j \geq a_0$, and for all $j' \in A_1$ with $j < j'$, the order of the points $\pi_j(\mathbf{z}_1),\ldots,\pi_j(\mathbf{z}_i)$ (in the sense of their relative positions in $I$) is the same as the order of the points $\pi_{j'}(\mathbf{z}_1),\ldots,\pi_{j'}(\mathbf{z}_i)$, meaning that
\begin{enumerate}[resume*=rec_prime]
\item $g_j^{j'}$ is order-preserving on the set $\pi_{j'}(\mathbf{Z}_i)$.
\end{enumerate}
Let $A_1'$ be the set of all $a \in A_1$ such that $g_a^{a^+}$ is decreasing at $\pi_{a^+}(\mathbf{z}_i)$, where $a^+$ denotes the smallest element of $A_1$ which is greater than $a$.  If $A_1'$ is finite, then by removing finitely many elements of $A_1$, we obtain a subset $A_2 \subseteq A_1$ such that for each consecutive pair $j,j'$ in $A_2$ with $j < j'$, $g_j^{j'}$ is increasing at $\pi_{j'}(\mathbf{z}_i)$.  On the other hand, if $A_1'$ is infinite, say $A_1' = \{a_1,a_2,\ldots\}$, where $a_1 < a_2 < \cdots$, then let $A_2 = \{a_1,a_3,a_5,\ldots\}$.  In either case, we obtain that
\begin{enumerate}[resume*=rec_prime]
\item for all $j,j' \in A_2$ with $j < j'$, $g_j^{j'}$ is increasing at $\pi_{j'}(\mathbf{z}_i)$.
\end{enumerate}
Let $p \in \{1,\ldots,i\}$ be such that $\pi_j(\mathbf{z}_p)$ is the largest element of $\pi_j(\mathbf{Z}_i)$ for each $j \in A_2$.  We consider two cases: either
\begin{enumerate}[label=(\alph{*})]
\item there exists $j \in A_2$ such that for all $j' \in A_2$ with $j' > j$, the map $g_j^{j'}$ is increasing on $[\pi_{j'}(\mathbf{z}_p),1]$; or
\item for each $j \in A_2$, there exists $j' \in A_2$ such that $j' > j$ and $g_j^{j'}([\pi_{j'}(\mathbf{z}_p),1]) = I$.
\end{enumerate}
If case (a) holds, then let $j_i(1)$ be any such $j$ as in (a).  If (b) holds, then let $j_i(1) \in A_2$ be arbitrary.  In the case that (b) holds, we point out that, as in the proof of Corollary~\ref{cor:comp extrema}, for any $N \in \mathbb{N}$, if $j' \in A_2$ is large enough, then there are at least $N$ extrema in $g_{j_i(1)}^{j'}$ between $\pi_{j'}(\mathbf{z}_p)$ and $1$.

According to Corollary~\ref{cor:comp extrema}, we may choose $j_i(2) \in \{j \in A_2: j > j_i(1)\}$ large enough so that
\begin{enumerate}[resume*=rec_prime]
\item for each two distinct points $z,z' \in \pi_{j_i(2)}(\mathbf{Z}_i)$, $g_{j_i(1)}^{j_i(2)}$ has at least $4i$ extrema between $z$ and $z'$.
\end{enumerate}
and also so that
\begin{enumerate}[resume*=rec_prime]
\item either (in case (a) above) $g_{j_i(1)}^{j_i(2)}$ is increasing on $[\pi_{j_i(2)}(\mathbf{z}_p),1]$, or (in case (b) above) $g_{j_i(1)}^{j_i(2)}$ has at least $2i$ extrema between $\pi_{j_i(2)}(\mathbf{z}_p)$ and $1$.
\end{enumerate}
Finally, let
\[ J_i = \{j_i(1)\} \cup \{j \in A_2: j \geq j_i(2)\} .\]

Properties (1)--(4) above for $J_i$ follow directly from the construction and properties (1$'$)--(4$'$).  Property (5) also follows from property (5$'$), since for any $j \in J_i$ with $j > j_i(1)$, we have $j \geq j_i(2)$, and so $g_{j_i(1)}^j = g_{j_i(1)}^{j_i(2)} \circ g_{j_i(2)}^j$, and for each two distinct points $z,z' \in \pi_j(\mathbf{Z}_i)$, there are at least as many extrema of $g_{j_i(1)}^j$ between $z$ and $z'$ as there are of $g_{j_i(1)}^{j_i(2)}$ between $g_{j_i(2)}^j(z)$ and $g_{j_i(2)}^j(z')$, which is at least $4i$ by (5$'$).  Property (6) similarly follows from property (6$'$), because for any $j \in J_i$ with $j \geq j_i(2)$ and each $z' = \pi_j(\mathbf{z}) \in \pi_j(\mathbf{Z}_i)$, $1 \in g_{j_i(2)}^j([z',1])$ due to property (4), hence $[z,1] \subseteq g_{j_i(2)}^j([z',1])$, where $z = g_{j_i(2)}^j(z') = \pi_{j_i(2)}(\mathbf{z})$.  This completes the recursive construction.
\end{proof}

Let $J = \{j_i(1): i \in \mathbb{N}\}$, and denote $j(i) = j_i(1)$, so that $J = \{j(1),j(2),\ldots\}$.  Recall that for each $k \in \mathbb{N}$, $g_k = T_{m_k}$ for some $m_k \geq 1$.  Observe that for any $i \in \mathbb{N}$, $g_{j(i)}^{j(i+1)} = T_{M_i}$, where $M_i = \prod_{k=j(i)}^{j(i+1)-1} m_k$.  By properties (4)--(6) above, we have that for each $i \in \mathbb{N}$, the map $g_{j(i)}^{j(i+1)}$ together with the set $Z = \pi_{j(i+1)}(\mathbf{Z}_i)$ satisfies the assumptions for Definition~\ref{def:Tm factorization} (using $m = M_i$ and $n = i$), so we may construct the map $s_i = s_{M_i,Z}$ as defined there, with the property that $T_{2i-1} \circ s_i = g_{j(i)}^{j(i+1)}$ (Lemma~\ref{lem:Tm factorization}).

According to Lemma~\ref{lem:s images} (using $m = M_{i+1}$, $Z = \pi_{j(i+2)}(\mathbf{Z}_{i+1})$, and $s = s_{i+1}$), $s_{i+1}$ preserves the order on the set $\pi_{j(i+2)}(\mathbf{Z}_{i+1})$, and also $T_{2i+1}$ is increasing at each point of $s_{i+1} \circ \pi_{j(i+2)}(\mathbf{Z}_{i+1})$.  By property (3), the map $g_{j(i+1)}^{j(i+2)} = T_{2i+1} \circ s_{i+1}$ also preserves the order on the set $\pi_{j(i+2)}(\mathbf{Z}_{i+1})$, and therefore $T_{2i+1}$ preserves the order on the set $s_{i+1} \circ \pi_{j(i+2)}(\mathbf{Z}_{i+1})$.  In particular, if we let $Z' = s_{i+1} \circ \pi_{j(i+2)}(\mathbf{Z}_i)$, then $T_{2i+1}$ preserves the order on $Z'$, and is increasing at each point of $Z'$.  Furthermore,
\[ T_{2i+1}(Z') = T_{2i+1} \circ s_{i+1} \circ \pi_{j(i+2)}(\mathbf{Z}_i)  = g_{j(i+1)}^{j(i+2)} \circ \pi_{j(i+2)}(\mathbf{Z}_i) = \pi_{j(i+1)}(\mathbf{Z}_i) .\]
Thus, we may apply Lemma~\ref{lem:s Tell visors removable} (using $m = M_i$, $Z = \pi_{j(i+1)}(\mathbf{Z}_i)$, $s = s_i$, $\ell = 2i+1$, and $Z' = s_{i+1} \circ \pi_{j(i+2)}(\mathbf{Z}_i)$) to conclude that in $s_i \circ T_{2i+1}$, every $s_{i+1} \circ \pi_{j(i+2)}(\mathbf{Z}_i)$-visor is removable.

For each $i \in \mathbb{N}$, define $f_i = s_i \circ T_{2i+1}$.  Then $K \approx \varprojlim \left\langle I,f_i \right\rangle$.  Denote by $\psi_i: K \to I$ the $i$-th coordinate projection with respect to this inverse limit representation of $K$, for each $i \in \mathbb{N}$.  Note that $\psi_i = s_i \circ \pi_{j(i+1)}$.  So from the previous paragraph, we have that every $\psi_{i+1}(\mathbf{Z}_i)$-visor under $f_i$ is removable.  Therefore, by Proposition~\ref{prop:embed ctble set accessible}, there is an embedding $\Omega \colon K \to \mathbb{R}^2$ such that for every $n \in \mathbb{N}$, $\Omega(\mathbf{z}_n)$ is an accessible point of $\Omega(K)$, as desired.
\end{proof}

In our final result, Example~\ref{ex:knaster attractor}, we present more new embeddings of the Knaster buckethandle continuum, which have dynamical significance.  We first recall the notion of a (compact) global attractor of a plane homeomorphism, following e.g.\ \cite{gobbino2001}.

Given a plane homeomorphism $H \colon \mathbb{R}^2 \to \mathbb{R}^2$, a compact set $A \subset \mathbb{R}^2$ is a \emph{global attractor} of $H$ if $A$ is \emph{invariant} and $A$ \emph{attracts} each bounded subset of $\mathbb{R}^2$; that is, if $H(A) = A$, and for any bounded set $P \subset \mathbb{R}^2$ and any $\varepsilon > 0$, there exists a non-negative integer $N$ such that for all $k \geq N$, $H^k(P)$ is contained in the $\varepsilon$-neighborhood of $A$.  If a plane homeomorphism has a (compact) global attractor, then it must be unique, hence we refer to it as \emph{the} global attractor.  Obviously, for an invariant compact set $A$ to be the global attractor of a plane homeomorphism, it suffices that $A$ attracts each compact subset of $\mathbb{R}^2$.

\begin{example}
\label{ex:knaster attractor}
Let $K = \underleftarrow{\lim}\langle I, T_2\rangle$ be the standard Knaster buckethandle continuum, and let $n \in \mathbb{N}$.  There exists a plane embedding $\Omega$ of $K$ such that at least $n$ composants of $\Omega(K)$ are accessible, and such that $\Omega(K)$ is the global attractor of a plane homeomorphism $H \colon \mathbb{R}^2 \to \mathbb{R}^2$ such that $H {\restriction}_{\Omega(K)}$ is conjugate to an iterate of the standard shift homeomorphism $\sigma \colon K \to K$ defined by
\[ \sigma(\langle x_1,x_2,x_3,\ldots \rangle) = \langle T_2(x_1),x_1,x_2,\ldots \rangle .\]
\end{example}

\begin{proof}
Let $k \in \mathbb{N}$ be large enough so that $2^{k-1} \geq 1 + \sum_{i=1}^n (2i-1)$.  Choose points $0 < a_0 < a_1 < \cdots < a_n \leq 1$ such that for each $i \in \{1,\ldots,n\}$, the interval $P_i = [a_{i-1},a_i]$ is a sawtooth pattern of $T_2^k = T_{2^k}$ with $2i-1$ teeth and height $1$.  For each $i \in \{1,\ldots,n\}$, let $z_i \in P_i$ be the fixed point of $T_{2^k}$ which is on the first half (the increasing part) of the $i$-th tooth of $P_i$.  Then $m = 2^k$ and $Z = \{z_1,\ldots,z_n\}$ satisfy the conditions of Definition~\ref{def:Tm factorization}, so we may define $s = s_{2^k,Z}$ as given there, so that $T_{2^k} = T_{2n-1} \circ s$ (Lemma~\ref{lem:Tm factorization}).  See Figure~\ref{fig:knaster ex} for an illustration for the case $n=2$.

\begin{figure}
\begin{center}
\includegraphics{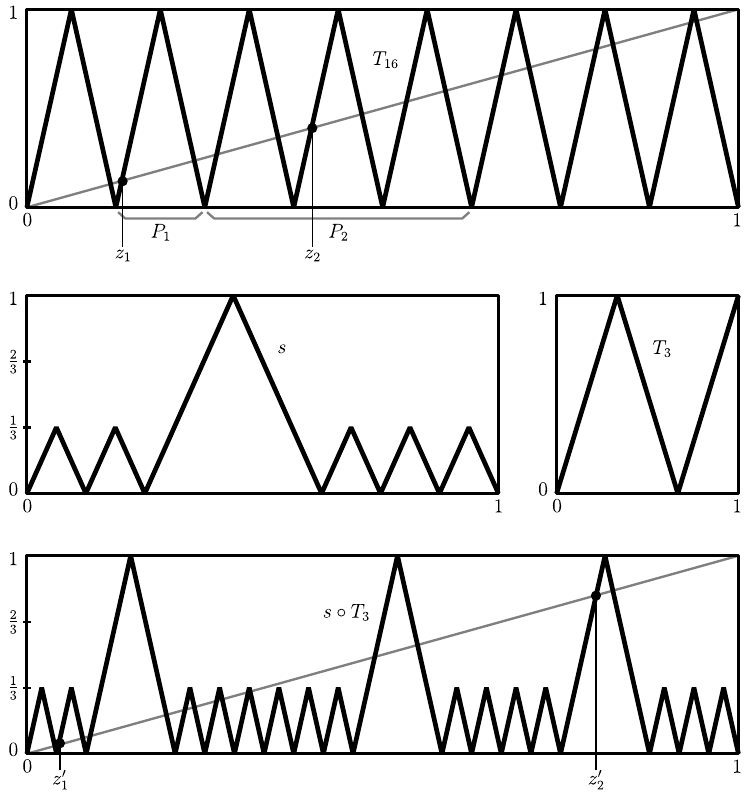}
\end{center}

\caption{First row: the tent map $T_{16}$, with a set $Z = \{z_1,z_2\}$ indicated, together with sawtooth patterns $P_1$ and $P_2$ as in Definition~\ref{def:Tm factorization}.  Second row: The map $s = s_{16,Z}$ from Definition~\ref{def:Tm factorization}, and the tent map $T_3$, which satisfy $T_{16} = T_3 \circ s$.  Third row: the composition $s \circ T_3$, with the set $Z' = \{z_1',z_2'\}$ as in the proof of Example~\ref{ex:knaster attractor} indicated.}
\label{fig:knaster ex}
\end{figure}

For each $i \in \{1,\ldots,n\}$, let $\mathbf{z}_i = \langle z_i,z_i,\ldots \rangle \in \varprojlim \left\langle I,T_{2^k} \right\rangle \approx K$.  Note that no two of these points $\mathbf{z}_i$ belong to the same composant of $K$, by Lemma~\ref{lem:Knaster composants}.

Let $f = s \circ T_{2n-1}$, and denote $K' = \varprojlim \left\langle I,f \right\rangle$.  From Lemma~\ref{lem:Tm factorization}, it follows that $K \approx K'$; in fact, an explicit homeomorphism $h \colon \varprojlim \left\langle I,T_{2^k} \right\rangle \to K'$ is given by
\[ h(\langle x_1,x_2,\ldots \rangle) = \langle s(x_1),s(x_2),\ldots \rangle .\]
Observe that the shift map $\widehat{f}$ on $K'$ induced by $f$ is conjugate (via $h$) to the $k$-th iterate of the shift map $\sigma$ on $K$ induced by $T_2$:
\begin{align*}
\widehat{f} \circ h(\langle x_1,x_2,\ldots \rangle) &= \langle f \circ s(p_1),f \circ s(p_2),\ldots \rangle) \\
&= \langle s \circ T_2^k(p_1),s \circ T_2^k(p_2),\ldots \rangle) \\
&= h \circ \sigma^k(\langle p_1,p_2,\ldots \rangle) .
\end{align*}

For each $i \in \{1,\ldots,n\}$, define $z_i' = s(z_i)$.  Let $Z' = \{z_1',\ldots,z_n'\}$, and for each $i \in \{1,\ldots,n\}$, denote $\mathbf{z}_i' = \langle z_i',z_i',\ldots \rangle$.  Notice that $h(\mathbf{z}_i) = \mathbf{z}_i'$ for each $i \in \{1,\ldots,n\}$.  By Lemma~\ref{lem:s images}, $z_1' < \cdots < z_n'$, and for each $i \in \{1,\ldots,n\}$, $T_{2n-1}$ is increasing at $z_i'$.  Also,
\[ T_{2n-1}(z_i') = T_{2n-1} \circ s(z_i) = T_{2^k}(z_i) = z_i .\]
So, by Lemma~\ref{lem:s Tell visors removable}, all $Z'$-visors under $f$ are removable.  At this point, we know by Proposition~\ref{prop:embed n pts accessible} that there is an embedding $\Omega' \colon K' \to \mathbb{R}^2$ such that for each $i \in \{1,\ldots,n\}$, $\Omega'(\mathbf{z}_i')$ is an accessible point of $\Omega'(K')$.  In any such embedding, the $n$ composants containing $\mathbf{z}_1',\ldots,\mathbf{z}_n'$ are fully accessible by \cite[Lemma A]{debski-tymchatyn1993}.  The remainder of the proof is devoted to describing such an embedding which is the global attractor of a plane homeomorphism.

We will closely follow related constructions from \cite{barge-martin1990} and \cite{boyland-carvalho-hall2021}.  We construct a closed topological disk $\mathbb{D}$ as follows.  Consider the standard unit circle $\mathbb{S}^1 = \{\langle x,y \rangle \in \mathbb{R}^2: x^2 + y^2 = 1\}$, and define the equivalence relation $\sim$ on the annulus $\mathbb{S}^1 \times [0,1]$ by $\langle x_1,y_1,t_1 \rangle \sim \langle x_2,y_2,t_2 \rangle$ if and only if $\langle x_1,y_1,t_1 \rangle = \langle x_2,y_2,t_2 \rangle$, or, $t_1 = t_2 = 1$ and $y_1 = y_2$.  Let $\mathbb{D} = \left( \mathbb{S}^1 \times [0,1] \right) / {\sim}$.

Let $J = \left( \mathbb{S}^1 \times \{1\} \right) / {\sim}$, and let $\mathbb{D}' = \left( \mathbb{S}^1 \times [\frac{1}{2},1] \right) / {\sim}$.  Then $J$ is an arc, and in fact an explicit homeomorphism $\tau \colon I \to J$ is given by
\[ \tau(p) = \left[ \left\langle \sqrt{1-(2p-1)^2},2p-1,1 \right\rangle \right] .\]
Here $[\langle x,y,t \rangle]$ denotes the $\sim$-equivalence class of $\langle x,y,t \rangle$, considered as a point in the quotient $\mathbb{D}$.  Also, $\mathbb{D}'$ is a closed topological disk, and $J \subset \intr(\mathbb{D}') \subset \mathbb{D}' \subset \intr(\mathbb{D})$.  For each $\theta \in \mathbb{S}^1$, let $\alpha_\theta = \left( \{\theta\} \times [0,1] \right) / {\sim}$, which is an arc in $\mathbb{D}$.  Let $\mathcal{A} = \{\alpha_\theta: \theta \in \mathbb{S}^1\}$.  See Figure~\ref{fig:knaster ex disk1} for an illustration.

\begin{figure}
\begin{center}
\includegraphics[width=4in]{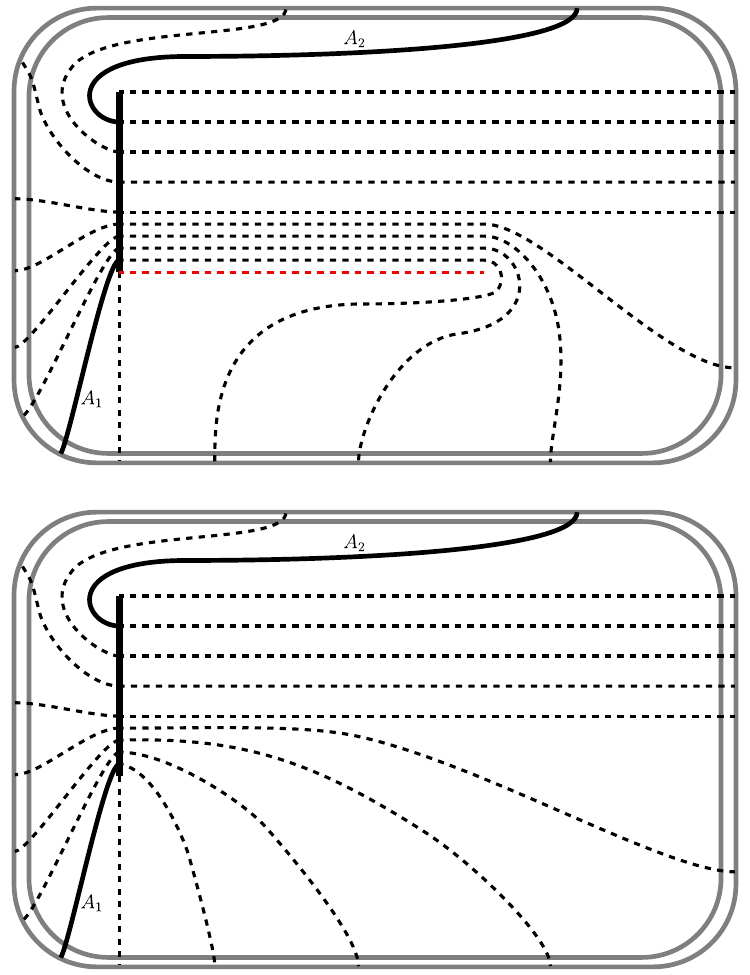}
\end{center}

\caption{Illustration of the disk $\mathbb{D}$ and family $\mathcal{A}$ of arcs in the bottom picture, and the pre-image under the map $H_2 \colon \mathbb{D} \to \mathbb{D}$ in the top picture.  The map $H_2$ collapses the red segment in the top picture to the point $\tau(0) \in J$, and is the identity on the vertical segment $J$, on the outer annulus $\mathbb{D} \smallsetminus \mathbb{D}'$, and on the two arcs $A_1$ and $A_2$.  The map $H_1 \colon \mathbb{D} \to \mathbb{D}$ is the identity on $\partial \mathbb{D}$, collapses the inner disk $\mathbb{D}'$ to the vertical segment $J$, and leaves the arcs of $\mathcal{A}$ set-wise invariant.  This particular drawing of $\mathbb{D}$, with $\mathcal{A}$ and $\mathbb{D}'$ in such shape, was chosen so as to simplify the picture in Figure~\ref{fig:knaster ex disk2} below.}
\label{fig:knaster ex disk1}
\end{figure}

Define $H_1 \colon \mathbb{D} \to \mathbb{D}$ by $H_1([\langle x,y,t \rangle]) = [\langle x,y,\min\{2t,1\} \rangle]$.  The key features of $H_1$ are:
\begin{enumerate}
\item $H_1$ is the identity on $J$ and on $\partial \mathbb{D}$;
\item $H_1(\mathbb{D}') = J$;
\item For each $\alpha \in \mathcal{A}$, $H_1(\alpha) = \alpha$;
\item For every compact set $P \subset \intr(\mathbb{D})$, there exists $N \in \mathbb{N}$ such that $H_1^N(P) \subseteq J$; and
\item $H_1$ is a near-homeomorphism.
\end{enumerate}

Next, let
\begin{align*}
\mathbb{D}_+ &= \{[\langle x,y,t \rangle] \in \mathbb{D}: x \geq 0\} \textrm{, and } \\
\mathbb{D}_- &= \{[\langle x,y,t \rangle] \in \mathbb{D}: x \leq 0\} .
\end{align*}
Observe that $J \subset \mathbb{D}_+ \cap \mathbb{D}_-$, and for each $\alpha \in \mathcal{A}$, either $\alpha \subset \mathbb{D}_+$ or $\alpha \subset \mathbb{D}_-$.  Also, for any $p \in I \smallsetminus \{0,1\}$, there exist exactly two arcs $\alpha_+,\alpha_- \in \mathcal{A}$ containing $\tau(p)$, one contained in $\mathbb{D}_+$ and one contained in $\mathbb{D}_-$; and for each $p \in \{0,1\}$, there exists exactly one arc $\alpha \in \mathbb{A}$ containing $\tau(p)$, and $\alpha \subset \mathbb{D}_+ \cap \mathbb{D}_-$.  For each $i \in \{1,\ldots,n\}$, let $A_i$ be the arc $\alpha \in \mathcal{A}$ such that $\tau(z_i') \in A_i$ and $A_i \subset \mathbb{D}_-$.

Let $H_2 \colon \mathbb{D} \to \mathbb{D}$ be a map such that
\begin{enumerate}
\item $H_2^{-1}(\tau(0))$ is a non-degenerate arc in $\mathbb{D}_+ \cap \intr(\mathbb{D}')$;
\item $H_2$ is one-to-one on $\mathbb{D} \smallsetminus H_2^{-1}(\tau(0))$;
\item $H_2$ is the identity on $J$, on each of the arcs $A_1,\ldots,A_n$, and on $\mathbb{D} \smallsetminus \mathbb{D}'$;
\item $H_2$ is a near-homeomorphism.
\end{enumerate}
See Figure~\ref{fig:knaster ex disk1} for an illustration.

Let $H_3 \colon \mathbb{D} \to \mathbb{D}$ be a homeomorphism such that:
\begin{enumerate}
\item $H_3(J) \subseteq \mathbb{D}_+ \cap \mathbb{D}'$;
\item $H_3(\tau(p)) = \tau(f(p))$ for every $p \in I$; and
\item $H_3$ is the identity on each of the arcs $A_1,\ldots,A_n$, and on $\mathbb{D} \smallsetminus \mathbb{D}'$.
\end{enumerate}
To construct such a map $H_3$, one may mimic the construction from Lemma~\ref{lem:tuck}.  See Figure~\ref{fig:knaster ex disk2} for an illustration of the map $H_3$ in the case $n=2$.

\begin{figure}
\begin{center}
\includegraphics[width=4in]{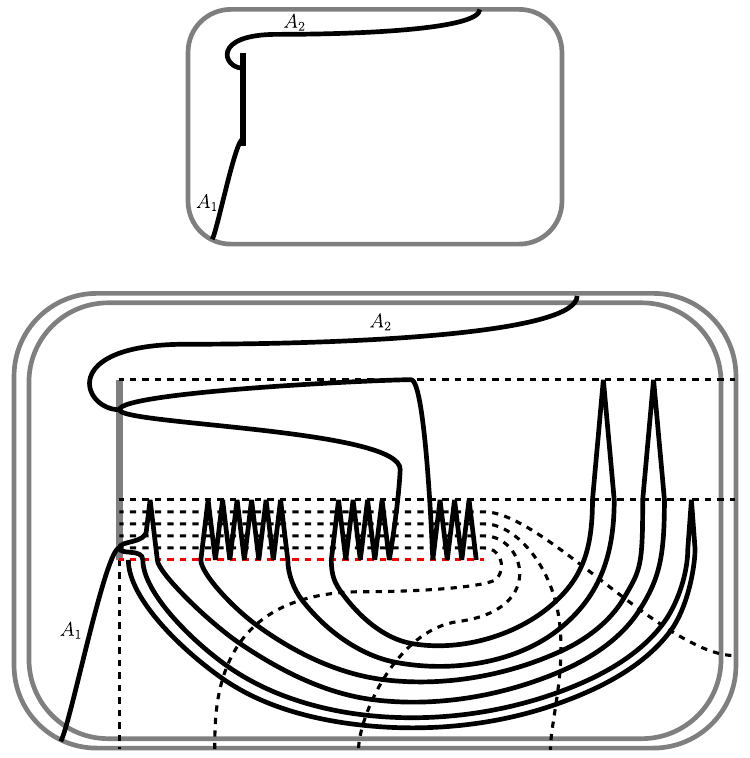}
\end{center}

\caption{Illustration, for the case $n=2$, of the map $H_3 \colon \mathbb{D} \to \mathbb{D}$, which is the identity on the arcs $A_1$ and $A_2$ and on the outer annulus $\mathbb{D} \smallsetminus \mathbb{D}'$, and maps the vertical segment $J$ in the top picture to the solid black arc in the bottom picture.}
\label{fig:knaster ex disk2}
\end{figure}

Let $H = H_1 \circ H_2 \circ H_3$.  Since $H_1$ and $H_2$ are near-homeomorphisms and $H_3$ is a homeomorphism, $H$ is a near-homeomorphism as well.  Therefore, by \cite{brown1960}, the space $\mathbb{D}^\infty = \varprojlim \left\langle \mathbb{D},H \right\rangle$ is a closed topological disk.  Moreover, for each $i \in \{1,\ldots,n\}$, the restriction $H {\restriction}_{A_i}$ is a near-homeomorphism of $A_i$ to itself, hence the space $A_i^\infty = \varprojlim \left\langle A_i,H {\restriction}_{A_i} \right\rangle$ is an arc contained in $\mathbb{D}^\infty$.

Since $H_1$ and $H_2$ are each the identity on $J$, by property (2) of $H_3$, we have that $J$ is invariant under $H$, and $H {\restriction}_J = \tau \circ f \circ \tau^{-1}$.  This means that $K^\infty = \varprojlim \left\langle J,H {\restriction}_J \right\rangle$ is a subspace of $\mathbb{D}^\infty$ which is homeomorphic to $K' = \varprojlim \left\langle I,f \right\rangle \approx K$, and an explicit homeomorphism $\tau^\infty \colon K' \to K^\infty$ is given by $\tau^\infty(\langle p_1,p_2,\ldots \rangle) = \langle \tau(p_1),\tau(p_2),\ldots \rangle$.  Observe that for each $i \in \{1,\ldots,n\}$, $A_i^\infty \cap K^\infty = \{\tau^\infty(\mathbf{z_i'})\}$; thus, $\tau^\infty(\mathbf{z_i'})$ is an accessible point of $K^\infty$.

The shift homeomorphism $\widehat{H}$ on $\mathbb{D}^\infty$ induced by $H$ is given by
\[ \widehat{H}(\langle q_1,q_2,q_3,\ldots \rangle) = \langle H(q_1),H(q_2),H(q_3),\ldots \rangle = \langle H(q_1),q_1,q_2,\ldots \rangle .\]
Clearly $K^\infty$ is invariant under $\widehat{H}$, and when restricted to $K^\infty$, $\widehat{H}$ is conjugate (via $\tau^\infty$) to the shift homeomorphism $\widehat{f}$ on $K'$ induced by $f$:
\begin{align*}
\widehat{H} \circ \tau^\infty(\langle p_1,p_2,\ldots \rangle) &= \langle H \circ \tau(p_1),H \circ \tau(p_2),\ldots \rangle) \\
&= \langle \tau \circ f(p_1),\tau \circ f(p_2),\ldots \rangle) \\
&= \tau^\infty \circ \widehat{f}(\langle p_1,p_2,\ldots \rangle) .
\end{align*}

Notice that, because the boundary of $\mathbb{D}$ is fully invariant under $H$, the interior of the closed topological disk $\mathbb{D}^\infty$ is $\intr(\mathbb{D}^\infty) = \varprojlim \left\langle \intr(\mathbb{D}),H {\restriction}_{\intr(\mathbb{D})} \right\rangle$, which is homeomorphic to $\mathbb{R}^2$.  In this way, we view $\widehat{H} {\restriction}_{\intr(\mathbb{D}^\infty)}$ as a plane homeomorphism.  To prove that $K^\infty$ is the global attractor of this plane homeomorphism, suppose that $P^\infty$ is any compact set contained in $\intr(\mathbb{D}^\infty)$.  Then for any $j \in \mathbb{N}$, the $j$-th coordinate projection of $P^\infty$ in $\mathbb{D}$ is a compact set $P_j$ which is contained in $\intr(\mathbb{D})$.  Therefore, there exists $N \in \mathbb{N}$ such that $H^N(P_j) \subseteq J$.  It follows that for any $\varepsilon > 0$, there exists $N \in \mathbb{N}$ such that $\widehat{H}^k(P^\infty)$ is within $\varepsilon$ from $K^\infty$ for all $k \geq N$.  This implies that $K^\infty$ is the global attractor of $\widehat{H} {\restriction}_{\intr(\mathbb{D}^\infty)}$.
\end{proof}

We remark that a similar construction as in Example~\ref{ex:knaster attractor} may be carried out for any Knaster continuum of the form $\varprojlim \left\langle I,T_m \right\rangle$, for $m \in \mathbb{N}$, $m \geq 2$.

\section{Discussion and questions}
\label{sec:discussion}

Propositions~\ref{prop:embed n pts accessible} and \ref{prop:embed ctble set accessible} may be used to produce plane embeddings of arc-like continua with certain points made accessible, but they do not offer any control over which other points might be made accessible as well.  However, in particular examples, for instance with Knaster continua, it seems evident that with careful choices of embeddings $\Phi$ of $I$ in the plane, playing the role of the embedding $\Phi$ constructed in Lemma~\ref{lem:tuck} and used in the proofs of Propositions~\ref{prop:embed n pts accessible} and \ref{prop:embed ctble set accessible}, one may construct embeddings in which exactly $n$ composants are fully accessible, and no other points are accessible.  See Figure~\ref{fig:knaster alt tuck} for an illustration of one such alternative embedding $\Phi$.  We leave this open as a question:

\begin{figure}
\begin{center}
\includegraphics{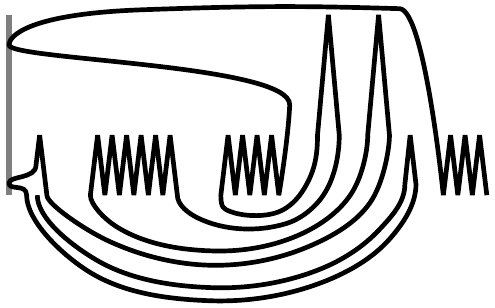}
\end{center}

\caption{An example of an alternative embedding $\Phi$ of $I$ into $\mathbb{H}$ which, if used to construct an embedding of the Knaster continuum $K' = \protect\varprojlim \left\langle I,f \right\rangle$ (where $f$ is as in Example~\ref{ex:knaster attractor} above), may limit which composants are accessible.}
\label{fig:knaster alt tuck}
\end{figure}

\begin{question}
Given any $n$ composants of a Knaster continuum, does there exist a plane embedding in which these $n$ composants are fully accessible, and no other points are accessible?
\end{question}

It would be interesting to know if the construction in Example~\ref{ex:knaster attractor} could be extended to involve countably infinitely many composants:

\begin{question}
Does there exist a plane embedding $\Omega$ of a Knaster continuum $K$ such that countably infinitely many composants of $\Omega(K)$ are accessible, and such that $\Omega(K)$ is the global attractor of a plane homeomorphism $H \colon \mathbb{R}^2 \to \mathbb{R}^2$?
\end{question}

Our application of Propositions~\ref{prop:embed n pts accessible} and \ref{prop:embed ctble set accessible} in this paper is limited to Knaster continua.  Lewis \cite{lewis1981} proved that for any finite subset $Z$ of the pseudo-arc $P$, there exists a plane embedding $\Omega$ of $P$ such that each point of $\Omega(Z)$ is an accessible point of $\Omega(P)$.  It remains an open question whether analagous results can be proved for arbitrary arc-like continua.

\begin{question}
Is it true that for any arc-like continuum $X$ and any finite (respectively, countably infinite) set $Z \subset X$, there exists an embedding $\Omega \colon X \to \mathbb{R}^2$ for which each point of $\Omega(Z)$ is an accessible point of $\Omega(X)$?
\end{question}

\bibliographystyle{amsplain}
\bibliography{Visors}

\end{document}